\documentclass[a4paper,11pt]{amsart}
\usepackage{amstext,amsfonts,amsthm,graphicx,amssymb,amscd,epsfig}
\usepackage{amsmath}
\usepackage[ansinew]{inputenc}    
\usepackage{psfrag}
\usepackage{mathrsfs}
\usepackage{xypic}
\usepackage{graphicx}
\usepackage[all,knot,arc,import,poly]{xy}

\theoremstyle{definition}
\newtheorem{definition}{Definition}[section]
\newtheorem{ex}[definition]{Example}
\newtheorem{rem}[definition]{Remark}

\theoremstyle{plain}
\newtheorem{prop}[definition]{Proposition}
\newtheorem{lem}[definition]{Lemma}
\newtheorem{coro}[definition]{Corollary}
\newtheorem{teo}[definition]{Theorem}

\newfont{\bbb}{msbm10 scaled\magstephalf}     

\title{The geometry of corank $1$ surfaces in $\mathbb{R}^{4}$}

\author{P. Benedini Riul, R. Oset Sinha, M. A. S. Ruas}

\date{}

\address{Instituto de Ci\^encias Matem\'aticas e de Computa\c{c}\~ao - USP,
Av. Trabalhador s\~ao-carlense, 400 - Centro,
CEP: 13566-590 - S\~ao Carlos - SP, Brazil}

\email{benedini@usp.br, maasruas@icmc.usp.br}

\address{Departament de Matem\`{a}tiques,
Universitat de Val\`encia, Campus de Burjassot, 46100 Burjassot,
Spain}

\email{raul.oset@uv.es}

\thanks{Work of P. Benedini Riul supported by CAPES - PVE  88887.122685/2016-00}
\thanks{Work of R. Oset Sinha partially supported by DGICYT Grant MTM2015--64013--P}
\thanks{Work of M. A. S. Ruas partially supported by CNPq Grant 306306/2015--8 and FAPESP Grant 2014/00304--2}

\begin{document}

\begin{abstract}
We study the geometry of surfaces in $\mathbb{R}^{4}$ with corank $1$ singularities. For such surfaces the singularities are isolated and at each point we define the curvature parabola in the normal space. This curve codifies all the second order information of the surface. Also, using this curve we define asymptotic and binormal directions, the umbilic curvature and study the flat geometry of the surface. It is shown that we can associate to this singular surface a regular one in $\mathbb{R}^{4}$ and relate their geometry.
\end{abstract}

\maketitle

\section{Introduction}

The differential geometry of singular surfaces in $\mathbb{R}^{3}$
has been an object of interest in the past decades. Singularity
theory has performed a massive contribution to the development of
this area. The geometry of the cross-cap (or Whitney umbrella), for
instance, has been studied by many authors:
\cite{BruceWest,DiasTari,FukuiHasegawa,HHNUY,HHNSUY,BallesterosTari,OsetSinhaTari}.
Also, the cuspidal edge, the most simple type of wave front, appears
in many papers in which significant advances have been achieved:
\cite{KRSUY,MartinsSaji,MartinsSaji2,NUY,OsetSinhaTari2,SUY}.

In \cite{MartinsBallesteros}, the authors studied in depth the
geometry of surfaces in $\mathbb{R}^{3}$ with corank $1$
singularities. Inspired by the definition of the curvature ellipse
for regular surfaces in $\mathbb{R}^{4}$ (\cite{Little}), they defined
the curvature parabola, a plane curve that may degenerate into a
half-line, a line or even a point and whose trace lies in the normal
plane of the surface. This special curve carries all the second
order information of the surface at the singular point. The second
fundamental form is defined and later used to present the concepts
of asymptotic and binormal directions. An invariant called umbilic
curvature (invariant under the action of
$\mathcal{R}^{2}\times\mathcal{O}(3)$, the subgroup of $2$-jets of
diffeomorphisms in the source and linear isometries in the target)
is defined at points for which the curvature parabola degenerates.
Finally, all those constructions are used to obtain results
regarding the contact of the surface with planes and spheres.

When projecting orthogonally an immersed surface $S\subset\mathbb{R}^4$ to $\mathbb{R}^3$, the composition of the parametrisation of $S$ with the projection can be seen locally as a map germ $(\mathbb R^2,0)\rightarrow (\mathbb R^3,0)$. Furthermore, if the direction of projection is tangent to the surface, singularities appear. In particular, if the direction is asymptotic, singularities more degenerate than a cross-cap (or Whitney umbrella) appear (\cite{BruceNogueira}).

In \cite{Benedini/Sinha}, the first two authors relate the geometry of the surface in $\mathbb{R}^4$ to the geometry of the singular projected surface in $ \mathbb{R}^3$. In particular, they relate the curvature
ellipse of the immersed surface $S$ to the curvature parabola of its singular projection.

Similarly, singular surfaces in $\mathbb{R}^{4}$ appear naturally as
projections of regular surfaces in $\mathbb{R}^{5}$ along tangent
directions. They can also be taken as the image of a smooth map
(that may have singularities)
$f:U\subset\mathbb{R}^{2}\rightarrow\mathbb{R}^{4}$, where $U$ is an
open subset. Two surfaces parametrised by $\mathcal{A}$-equivalent
maps (two map germs,
$f,g:(\mathbb{R}^{2},q)\rightarrow(\mathbb{R}^{4},p)$ are
$\mathcal{A}$-equivalent if there are germs of diffeomorphisms
$\phi$ and $\psi$ of the source and target, respectively, such that
$g=\psi\circ f\circ \phi^{-1}$) have diffeomorphic images but may
not have the same local differential geometry. In \cite{Rieger}, the
authors give a classifications of all $\mathcal{A}$-simple map germs
$f:(\mathbb{R}^{2},0)\rightarrow(\mathbb{R}^{4},0)$. Since we are
interested in the geometry of the image of such map germs, we search
for a finer equivalence relation, one that preserves the geometric
properties of the surface.

In this paper, we present a study for surfaces in $\mathbb{R}^{4}$
with corank $1$ singularities. Section 2 is an overview of the
differential geometry of regular surfaces in $\mathbb{R}^{n}$,
$n=4,5$, for such surfaces will perform an important role. Let
$M\subset\mathbb{R}^{4}$ be a singular surface with a corank $1$
singularity at $p$. It will be considered as the image of a smooth
map $g:\tilde{M}\rightarrow\mathbb{R}^{4}$ with a singularity of
corank $1$ at $q$ and defined on a regular surface $\tilde{M}$, such
that $g(q)=p\in M$.

Section 3 is dedicated to the study of the \emph{curvature parabola} (see Definition \ref{curvatureparabola}). To do so, we define the first and second fundamental forms of $M$ at $p$. The curvature parabola, denoted by $\Delta_{p}$ is a plane curve whose trace may degenerate into a half-line, a line or a point. In \cite{Mendes}, the authors show that there are four orbits on the set of all corank $1$ map germs $f:(\mathbb{R}^{2},0)\rightarrow(\mathbb{R}^{4},0)$ according to their $2$-jets under the action of $\mathcal{A}^{2}$, which is the space of $2$-jets of diffeomorphisms in the source and target (see Lemma \ref{lema2jato}). The curvature parabola is a complete invariant for this classification, in the sense that it distinguishes completely the four orbits (see Theorem \ref{teo.2jato/parabola}). The main result of this section is Theorem \ref{mainteo}, which shows that two corank $1$ $2$-jets $(\mathbb{R}^{2},0)\rightarrow(\mathbb{R}^{4},0)$ are equivalent under the action of $\mathcal{R}^{2}\times\mathcal{O}(4)$ if and only if there is an isometry between the normal hyperplanes that preserves the respective curvature parabolas.

The last section presents the second order properties of corank $1$
surfaces in $\mathbb{R}^{4}$. The purpose of this section is to
prove a  geometrical relation between the singular surface and a
regular surface $S\subset\mathbb{R}^{4}$ associated to the singular
one (Theorem \ref{teorelation}). In order to find this relation, we
also need to associate to our singular surface a regular surface
$N\subset\mathbb{R}^{5}$. The concepts of \emph{asymptotic} and
\emph{binormal directions} are defined and we prove some results
concerning them. Our singular surface in $\mathbb{R}^{4}$ is also
related to the associated regular surface $N\subset\mathbb{R}^{5}$:
they have the same second fundamental form. However, the notions of
asymptotic and binormal directions are different.

Finally, a second order invariant called the \emph{umbilic
curvature} is defined. One can find in the literature invariants
called umbilic curvature, for instance, in
\cite{Costa/Fuster/Moraes,MartinsBallesteros,MartinsSaji}. Our
invariant can be seen as a generalization of the previous ones for
the case of singular surfaces in $\mathbb R^4$. This invariant will
be an important tool to develop the study of the flat geometry of
corank $1$ surfaces in $\mathbb{R}^{4}$, the subject of the last
subsection in this paper.


\section{The geometry of regular surfaces}\label{regsurfaces}

In this section we present some aspects of regular surfaces in $\mathbb{R}^{n}$ for $n=4,5$. For more details, see \cite{Livro}.

\subsection{Regular surfaces in $\mathbb{R}^{4}$}\label{supregularesR4}
Little, in \cite{Little}, studied the second order geometry of submanifolds immersed in Euclidean spaces, in particular of immersed surfaces in $\mathbb{R}^{4}$. He defined the second fundamental form and the locus of curvature of such surfaces: the curvature ellipse, as we shall see.  This paper has inspired
a lot of research on the subject (see \cite{BruceNogueira,BruceTari,GarciaMochidaFusterRuas,MochidaFusterRuas,MochidaFusterRuas2,BallesterosTari,OsetSinhaTari,RomeroFuster}, amongst others). Given a smooth surface $S\subset\mathbb{R}^{4}$ and $f:U\rightarrow\mathbb{R}^{4}$
a local parametrisation of $S$ with $U\subset\mathbb{R}^{2}$ an open subset, let
$\{\textbf{e}_{1},\textbf{e}_{2},\textbf{e}_{3},\textbf{e}_{4}\}$ be an orthonormal frame of $\mathbb{R}^{4}$ such that at any $u\in U$,
$\{\textbf{e}_{1}(u),\textbf{e}_{2}(u)\}$ is a basis for $T_{p}S$ and $\{\textbf{e}_{3}(u),\textbf{e}_{4}(u)\}$ is a basis for
$N_{p}S$ at $p=f(u)$.
The second fundamental form of $S$ at $p$ is the vector valued quadratic form
$II_{p}:T_{p}S\rightarrow N_{p}S$ given by
$$II_{p}(\textbf{w})=(l_{1}w_{1}^{2}+2m_{1}w_{1}w_{2}+n_{1}w_{2}^{2})\textbf{e}_{3}+(l_{2}w_{1}^{2}+2m_{2}w_{1}w_{2}+n_{2}w_{2}^{2})\textbf{e}_{4},$$
where $l_{i}=\langle f_{xx},\textbf{e}_{i+2}\rangle,\ m_{i}=\langle f_{xy},\textbf{e}_{i+2}\rangle$
and $n_{i}=\langle f_{yy},\textbf{e}_{i+2}\rangle$ for $i=1,2$ are
called the coefficients of the second fundamental form with respect to the
frame above and $\textbf{w}=w_{1}\textbf{e}_{1}+w_{2}\textbf{e}_{2}\in T_{p}S$. The matrix of the second fundamental
form with respect to the orthonormal frame above is given by
$$
\alpha=\left(
         \begin{array}{ccc}
           l_{1} & m_{1} & n_{1} \\
           l_{2} & m_{2} & n_{2} \\
         \end{array}
       \right).
$$

Consider a point $p\in S$ and the unit circle $\mathbb{S}^{1}$ in $T_{p}S$ parametrised by
$\theta\in [0,2\pi]$. The curvature vectors $\eta(\theta)$ of the normal sections of
$S$ by the hyperplane $\langle\theta\rangle\oplus N_{p}S$ form an ellipse in the
normal plane $N_{p}S$, called the \emph{curvature ellipse} of $S$ at $p$, that can also
be seen as the image of the map
$\eta:\mathbb{S}^{1}\subset T_{p}S\rightarrow N_{p}S$, where
\begin{equation}
\eta(\theta)=\sum_{i=1}^{2}(l_{i}\cos^{2}\theta+2m_{i}\cos\theta\sin\theta+n_{i}\sin^{2}\theta)\textbf{e}_{i+2}.
\end{equation}\label{ellipse}
Note that if we write $\textbf{u}=\cos\theta\textbf{e}_{1}+\sin\theta\textbf{e}_{2}\in \mathbb{S}^{1}$,
$II_{p}(\textbf{u})=\eta(\theta)$.

The classification of points is made using the curvature ellipse.

\begin{definition}
A point $p\in S$ is called \emph{semiumbilic} if the curvature ellipse is a line segment which
does not contain $p$. If the curvature ellipse is a radial segment, the point $p$ is called an \emph{inflection} point. An inflection point is of \emph{real type}, (resp. \emph{imaginary type}, \emph{flat}) if $p$ is an interior point of the
radial segment, (resp. does not belong to it, is one of its end points).
When the curvature ellipse reduces to a point, $p$ is called \emph{umbilic}. Moreover,
if the point is $p$ itself, then $p$ is said to be a
\emph{flat umbilic}. A non inflection point $p\in S$ is called \emph{elliptic} (resp. \emph{hyperbolic},
\emph{parabolic}) when it lies inside (resp. outside, on) the curvature ellipse.
\end{definition}

A tangent direction $\theta$ at $p\in S$ is called an \emph{asymptotic direction}
at $p$ if $\eta(\theta)$ and $\frac{d\eta}{d\theta}(\theta)$ are linearly dependent vectors in
$N_{p}S$, where $\eta(\theta)$ is a parametrisation of the curvature ellipse as in (\ref{ellipse}). A curve on $S$ whose tangent at each point is an asymptotic
direction is called an asymptotic curve.

\begin{lem}\cite{Livro}
Let $f:U\rightarrow\mathbb{R}^{4}$ be a local parametrisation of a surface $S$
and denote by $l_{1},\ m_{1},\ n_{1},\ l_{2},\ m_{2},\ n_{2}$ the coefficients
of its second fundamental form with respect to any frame
$\{f_{x},f_{y},\textbf{f}_{3},\textbf{f}_{4}\}$ of $T_{p}S\times N_{p}S$
which depends smoothly on $p=f(x,y)$. Then the asymptotic curves of $S$ are
are the solutions curves of the binary differential equation:
\begin{equation}\label{BDE}
(l_{1}m_{2}-l_{2}m_{1})dx^{2}+(l_{1}n_{2}-l_{2}n_{1})dxdy+(m_{1}n_{2}-m_{2}n_{1})dy^{2}=0,
\end{equation}
which can also be written as the following determinant form:
$$
\left|
  \begin{array}{ccc}
    dy^{2} & -dxdy & dx^{2} \\
    l_{1} & m_{1} & n_{1} \\
    l_{2} & m_{2} & n_{2} \\
  \end{array}
\right|=0.
$$
\end{lem}

In order to obtain geometrical information of a regular surface $S\subset\mathbb{R}^{4}$, one can study the generic contacts of the surface with hyperplanes. Such contact is measured
by the singularities of the height function of $S$. Let $f:U\rightarrow\mathbb{R}^{4}$ be a local parametrisation of $S$. The family of height functions $H:U\times \mathbb{S}^{3}\rightarrow\mathbb{R}$
is given by $H(u,v)=\langle f(u),v\rangle$.
For $v$ fixed, the height function $h_{v}$ of $S$ is given by $h_{v}(u)=H(u,v)$.
A point $p=f(u)$ is a singular point of $h_{v}$ if and only if $v$ is a normal vector to the surface $S$ at $p$.
A hyperplane orthogonal to the direction $v$ is an \emph{osculating
hyperplane} of $M$ at $p = f(u)$ if it is tangent to $S$ at $p$ and $h_{v}$ has a
degenerate (i.e., non Morse) singularity at $u$. In such case we call the direction $v$ a
\emph{binormal} direction of $S$ at $p$.


\subsection{Surfaces in $\mathbb{R}^{5}$} The definitions of the second fundamental form and the curvature ellipse are very simillar to the ones of regular surfaces in $\mathbb{R}^{4}$. For more details, see \cite{Costa/Fuster/Moraes,Livro,Mochida/Fuster/Ruas,Moraes/Fuster,Fuster/Ruas/Tari}.
Let $N$ be a regular surface in $\mathbb{R}^{5}$ and $f:U\subset\mathbb{R}^{2}\rightarrow\mathbb{R}^{5}$ a local parametrisation of $N$. The positively oriented orthonormal frame $\{\textbf{e}_{1},\textbf{e}_{2},\textbf{e}_{3},\textbf{e}_{4},\textbf{e}_{5}\}$ of $\mathbb{R}^{5}$ satisfies that for every $u\in U$, $\{\textbf{e}_{1}(u),\textbf{e}_{2}(u)\}$ is a basis for the tangent plane $T_{p}N$ and $\{\textbf{e}_{3}(u),\textbf{e}_{4}(u),\textbf{e}_{5}(u)\}$ is a basis for the normal hyperplane $N_{p}N$ at $p=f(u)$. The \emph{second fundamental form} associated to the embedding can be represented by
$$
\alpha_{p}=\left(
\begin{array}{ccc}
    a_{1} & b_{1} & c_{1} \\
    a_{2} & b_{2} & c_{2} \\
    a_{3} & b_{3} & c_{3}
\end{array}
\right),$$

\noindent where $a_{i}=\langle f_{xx},\textbf{e}_{i+2}\rangle$, $b_{i}=\langle f_{xy},\textbf{e}_{i+2}\rangle$ and $c_{i}=\langle f_{yy},\textbf{e}_{i+2}\rangle$, $i=1,2,3$.
The \emph{curvature ellipse} at a point $p$ of the surface $N$ is the image of the map $\eta:\mathbb{S}^{1}\rightarrow N_{p}N$, obtained by assigning to each tangent direction $\theta$ the curvature vector of the normal section $\gamma_{\theta}$.

We also define the subsets
$$N_{i}=\{p\in N|\ \mbox{rank}(\alpha_{p})=i\},\ i\leqslant3.$$
Mochida, Romero Fuster and Ruas in \cite{Mochida/Fuster/Ruas} showed that given a closed surface $N$, there exists a residual set $\mathcal{O}$ in $Emb(N,\mathbb{R}^{5})$ such that for any $f\in\mathcal{O}$, $N=N_{3}\cup N_{2}$. Moreover, it is shown that $N_{3}$ is an open subset of $N$ and $N_{2}$ is a regularly embedded curve.
The affine subspace and the vector subspace determined by the curvature ellipse in $N_{p}N$ are denoted, respectively, by $\mathcal{A}ff_{p}$ and $E_{p}$.

Given $N\subset\mathbb{R}^{5}$ locally given by $f$, the family of height functions on $N$ is
$$\begin{array}{ccc}
     H:N\times\mathbb{S}^{4} &\rightarrow & \mathbb{R}\times\mathbb{S}^{4}  \\
     (p,v)& \mapsto & (h_{v}(p),v)
\end{array}
$$
where $h_{v}(p)=\langle f(u),v\rangle$.
There is an open and dense set $\mathcal{S}_{H}$ in $Imm(U,\mathbb{R}^{5})$ such that for any $f\in\mathcal{S}_{H}$, the height function $h_{v}$, $v\in\mathbb{S}^{4}$, of the surface $N=f(U)$ has only singularities of type $A_{\leq5}$, $D^{\pm}_{4}$ or $D_{5}$ and they all are $\mathcal{R}$-versally unfolded by the family $H$.

A vector $v\in N_{p}N$ will be called a \emph{degenerate direction} for $N$ provided that $p$ is a non Morse singularity of $h_{v}$. In such case, $\ker(\mathcal{H}(h_{v})(p))\neq\{0\}$ and any direction $\textbf{u}\in\ker(\mathcal{H}(h_{v})(p))$ is called \emph{flat contact direction associated to} $v$, where $\mathcal{H}(h_{v})(p)$ denotes the Hessian matrix of $h_{v}$ at $p$.

A degenerate direction $v\in N_{p}N$ for which $h_{v}$ has a singularity $A_{3}$ or worse is called a \emph{binormal direction} at $p$ and the corresponding contact directions are called \emph{asymptotic directions} on $N$. The subset of degenerate directions in $N_{p}N$ is a cone that may degenerate. It is shown in \cite{Mochida/Fuster/Ruas} that at a point of type $M_{3}$ there is at least one and at most five binormal directions.


\section{The curvature parabola}\label{section-notation}

In Subsection \ref{supregularesR4} we presented a brief study of regular surfaces in $\mathbb{R}^{4}$.
However, to study singular surfaces in $\mathbb{R}^{4}$ we face some problems: How do we define the tangent and the normal spaces? How do we get the second order geometry of such surface? In order to answer these questions we shall need the following construction. This construction and the results that follow were inspired by \cite{MartinsBallesteros}.

Let $M$ be a corank $1$ surface in $\mathbb{R}^{4}$ at $p$. We will take $M$ as the image of a smooth map $g:\tilde{M}\rightarrow \mathbb{R}^{4}$, where $\tilde{M}$ is a smooth regular surface and $q\in\tilde{M}$ is a corank $1$ point of $g$ such that $g(q)=p$. Also, we consider $\phi:U\rightarrow\mathbb{R}^{2}$ a local coordinate system defined in an open neighbourhood $U$ of $q$ at $\tilde{M}$, and by doing this we may consider a local parametrisation $f=g\circ\phi^{-1}$ of $M$ at $p$ (see the diagram below).

$$
\xymatrix{
\mathbb{R}^{2}\ar@/_0.7cm/[rr]^-{f} & U\subset\tilde{M}\ar[r]^-{g}\ar[l]_-{\phi} & M\subset\mathbb{R}^{4}
}
$$


The previous construction is a formality to ensure that the surface $M$ is a corank $1$ surface at $p$. Besides, it helps us to answer the first question of how to define the tangent space of $M$ at $p$. The \emph{tangent line} of $M$ at $p$, $T_{p}M$, is given by $\mbox{Im}\ dg_{q}$, where $dg_{q}:T_{q}\tilde{M}\rightarrow T_{p}\mathbb{R}^{4}$ is the differential map of $g$ at $q$. Hence, the \emph{normal hyperplane} of $M$ at $p$, $N_{p}M$, is the subspace satisfying $T_{p}M\oplus N_{p}M=T_{p}\mathbb{R}^{4}$.

Even though the first question has been answered, the fundamental forms will not be defined on $T_{p}M$, since it is a line. Here, our construction shall be convenient once more.
First, consider the orthogonal projection $\perp:T_{p}\mathbb{R}^{4}\rightarrow N_{p}M$, $w\mapsto w^{\perp}$.
The \emph{first fundamental form} of $M$ at $p$, $I:T_{q}\tilde{M}\times T_{q}\tilde{M}\rightarrow \mathbb{R}$ is given by
$$I(\textbf{u},\textbf{v})=\langle dg_{q}(\textbf{u}),dg_{q}(\textbf{v})\rangle,\ \ \ \ \forall\ \textbf{u},\textbf{v}\in T_{q}\tilde{M}.$$
Since the map $g$ has corank $1$ at $q\in T_{q}\tilde{M}$, the first fundamental form is not a Riemannian metric on $T_{q}\tilde{M}$, but a pseudometric. Considering the local parametrisation of $M$ at $p$, $f=g\circ\phi^{-1}$ and the basis $\{\partial_{x},\partial_{y}\}$ of $T_{q}\tilde{M}$, the coefficients of the first fundamental form with respect to $\phi$ are:
$$\begin{array}{c}
     E(q)=I(\partial_{x},\partial_{x})=\langle f_{x},f_{x}\rangle(\phi(q)),\   F(q)=I(\partial_{x},\partial_{y})=\langle f_{x},f_{y}\rangle(\phi(q)), \\
      G(q)=I(\partial_{y},\partial_{y})=\langle f_{y},f_{y}\rangle(\phi(q)).
\end{array}$$
Taking $\textbf{u}=\alpha\partial_{x}+\beta\partial_{y}=(\alpha,\beta)\in T_{q}\tilde{M}$, we write $I(\textbf{u},\textbf{u})=\alpha^{2}E(q)+2\alpha\beta F(q)+\beta^{2}G(q)$.

With the same conditions as above, the \emph{second fundamental form} of $M$ at $p$, $II:T_{q}\tilde{M}\times T_{q}\tilde{M}\rightarrow N_{p}M$ in the basis $\{\partial_{x},\partial_{y}\}$ of $T_{q}\tilde{M}$ is given by
$$
\begin{array}{c}
     II(\partial_{x},\partial_{x})=f_{xx}^{\perp}(\phi(q)),\  II(\partial_{x},\partial_{y})=f_{xy}^{\perp}(\phi(q)),\
      II(\partial_{y},\partial_{y})=f_{yy}^{\perp}(\phi(q))
\end{array}
$$
and we extend it to the whole space in a unique way as a symmetric bilinear map.

\begin{lem}\label{lemma-2ff}
The definition of the second fundamental form does not depend on the choice of local coordinates on $\tilde{M}$.
\end{lem}
\begin{proof}
Let $\bar{\phi}:\bar{U}\subset\mathbb{R}^{2}\rightarrow\tilde{M}$ another local coordinate system of $\tilde{M}$ at $q\in U\cap\bar{U}$ with coordinates $(u,v)$ and denote by $\bar{f}=g\circ\bar{\phi}$ the corresponding local parametrisation of $M$ at $p$. The parcial derivatives of $\bar{f}$ are
$$\bar{f}_{u}=f_{x}x_{u}+f_{y}y_{u}\ \mbox{and}\ \bar{f}_{v}=f_{x}x_{v}+f_{y}y_{v},$$
hence, we have
$$\begin{array}{ccc}
     \bar{f}_{uu}=f_{x}x_{uu}+f_{y}y_{uu}+f_{xx}x_{u}^{2}+2f_{xy}x_{u}y_{u}+f_{yy}y_{u}^{2}; \\
     \bar{f}_{uv}=f_{x}x_{uv}+f_{y}y_{uv}+f_{xx}x_{u}x_{v}+f_{xy}(x_{u}y_{v}+x_{v}y_{u})+f_{yy}y_{u}y_{v}; \\
     \bar{f}_{vv}=f_{x}x_{vv}+f_{y}y_{vv}+f_{xx}x_{v}^{2}+2f_{xy}x_{v}y_{v}+f_{yy}y_{v}^{2}.
\end{array}$$

Projecting to $N_{p}M$, we obtain
$$\begin{array}{ccc}
     \bar{f}_{uu}^{\perp}=f_{xx}^{\perp}x_{u}^{2}+2f_{xy}^{\perp}x_{u}y_{u}+f_{yy}^{\perp}y_{u}^{2}; \\
     \bar{f}_{uv}^{\perp}=f_{xx}^{\perp}x_{u}x_{v}+f_{xy}^{\perp}(x_{u}y_{v}+x_{v}y_{u})+f_{yy}^{\perp}y_{u}y_{v}; \\
     \bar{f}_{vv}^{\perp}=f_{xx}^{\perp}x_{v}^{2}+2f_{xy}^{\perp}x_{v}y_{v}+f_{yy}^{\perp}y_{v}^{2}.
\end{array}$$
Since the projection is linear, the vectors $\bar{f}_{uu}^{\perp}$, $\bar{f}_{uv}^{\perp}$, $\bar{f}_{vv}^{\perp}$ and $f_{xx}^{\perp}$, $f_{xy}^{\perp}$, $f_{yy}^{\perp}$ are related by the equations of basis change in a symmetric bilinear map with respect to the matrix
$$\left(
\begin{array}{cc}
    x_{u} & x_{v} \\
    y_{u} & y_{v}
\end{array}
\right)$$
which is the matrix of basis change from $\{\partial_{x},\partial_{y}\}$ to $\{\partial_{u},\partial_{v}\}$ in $T_{q}\tilde{M}$. Therefore, both coordinates define the same second fundamental form.
\end{proof}

It is possible, for each normal vector $\nu\in N_{p}M$, to define the \emph{second fundamental form along $\nu$}, $II_{\nu}:T_{q}\tilde{M}\times T_{q}\tilde{M}\rightarrow\mathbb{R}$ given by $II_{\nu}(\textbf{u},\textbf{v})=\langle II(\textbf{u},\textbf{v}),\nu\rangle$, for all $\textbf{u},\textbf{v}\in T_{q}\tilde{M}$. The coefficients of $II_{\nu}$ with respect to the basis $\{\partial_{x},\partial_{y}\}$ of $T_{q}\tilde{M}$ are
$$
\begin{array}{cc}
     l_{\nu}(q)=\langle f_{xx}^{\perp},\nu\rangle(\phi(q)),\ m_{\nu}(q)=\langle f_{xy}^{\perp},\nu\rangle(\phi(q)),  \\
     n_{\nu}(q)=\langle f_{yy}^{\perp},\nu\rangle(\phi(q)).
\end{array}
$$
Given $\textbf{u}=\alpha\partial_{x}+\beta\partial_{y}\in T_{q}\tilde{M}$,
$$II_{\nu}(\textbf{u},\textbf{u})=\langle II(\textbf{u},\textbf{u}),\nu\rangle=\alpha^{2}l_{\nu}(q)+2\alpha\beta m_{\nu}(q)+\beta^{2}n_{\nu}(q).$$
Fixing an orthonormal frame $\{\nu_{1},\nu_{2},\nu_{3}\}$ of $N_{p}M$,
$$
\begin{array}{cl}\label{eq.2ff}
II(\textbf{u},\textbf{u}) & =II_{\nu_{1}}(\textbf{u},\textbf{u})\nu_{1}+II_{\nu_{2}}(\textbf{u},\textbf{u})\nu_{2}+II_{\nu_{3}}(\textbf{u},\textbf{u})\nu_{3} \\
        & =\displaystyle{\sum_{i=1}^{3}(\alpha^{2}l_{\nu_{i}}(q)+2\alpha\beta m_{\nu_{i}}(q)+\beta^{2}n_{\nu_{i}}(q))\nu_{i}},
\end{array}
$$
Moreover, the second fundamental form is represented by the matrix of coefficients
$$
\left(
  \begin{array}{ccc}
    l_{\nu_{1}} & m_{\nu_{1}} & n_{\nu_{1}} \\
    l_{\nu_{2}} & m_{\nu_{2}} & n_{\nu_{2}} \\
    l_{\nu_{3}} & m_{\nu_{3}} & n_{\nu_{3}} \\
  \end{array}
\right).
$$

\begin{definition}\label{curvatureparabola}
Let $C_{q}\subset T_{q}\tilde{M}$ be the subset of unit tangent vectors and let $\eta_{q}:C_{q}\rightarrow N_{p}M$ be the map given by $\eta_{q}(\textbf{u})=II(\textbf{u},\textbf{u})$. The \emph{curvature parabola} of $M$ at $p$, denoted by $\Delta_{p}$, is the image of $\eta_{q}$, that is, $\eta_{q}(C_{q})$.
\end{definition}

It follows from Lemma \ref{lemma-2ff} that the curvature parabola does not depend on the choice of local coordinates on the surface $\tilde{M}$. However, it depends on the map $g$ which parametrises $M$.

\begin{ex}
Consider $\tilde{M}=\mathbb{R}^{2}$ and the surface $M$ parametrised by $g(x,y)=(x,xy,y^{2},y^{2k+1})$, with $k\geqslant1$. Taking coordinates $(X,Y,Z,W)$ in $\mathbb{R}^{4}$, $q=(0,0)$ and $p=(0,0,0,0)$, the tangent line $T_{p}M$ is the $X$-axis and $N_{p}M$ is the $YZW$-hyperplane. The coefficients of the first fundamental form are given by $E(q)=1$ and $F(q)=G(q)=0$. Hence, if $\textbf{u}=(\alpha,\beta)\in T_{q}\mathbb{R}^{2}$, $I(\textbf{u},\textbf{u})=\alpha^{2}$ and $C_{q}=\{(\pm1,y):y\in\mathbb{R}\}$. The matrix of coefficients of the second fundamental form is
$$\left(
\begin{array}{ccc}
    0 & 1 & 0  \\
    0 & 0 & 2  \\
    0 & 0 & 0  \\
\end{array}
\right)
$$
when we consider the orthonormal frame $\{\textbf{e}_{1},\textbf{e}_{2},\textbf{e}_{3},\textbf{e}_{4}\}$. Therefore, for $\textbf{u}=(\alpha,\beta)$, $II(\textbf{u},\textbf{u})=(0,2\alpha\beta,2\beta^{2},0)$ and
the curvature parabola $\Delta_{p}$ is a non-degenerate parabola which can be parametrised by $\eta(y)=(0,2y,2y^{2},0)$.
\end{ex}

Let $g:(\mathbb{R}^{2},0)\rightarrow(\mathbb{R}^{4},0)$ be a corank $1$ map at $q$. It is possible to take a coordinate system $\phi$ and make rotations in the target in order to obtain
$$f(x,y)=g\circ\phi^{-1}(x,y)=(x,f_{2}(x,y),f_{3}(x,y),f_{4}(x,y)),$$
where $\frac{\partial f_{i}}{\partial x}(\phi(q))=\frac{\partial f_{i}}{\partial y}(\phi(q))=0$ for $i=2,3,4$. In this new parametrisation, we have $E(q)=1$, $F(q)=G(q)=0$ and if $\textbf{u}=(\alpha,\beta)$ is a tangent direction, $I(\textbf{u},\textbf{u})=\alpha^{2}$. Hence, $\textbf{u}\in C_{q}$ iff $\alpha=\pm1$ and $\beta\in\mathbb{R}$. Geometrically, $C_{q}$ is a pair of lines parallel to the $\partial_{y}$-axis in the tangent plane. Taking an orthonormal frame $\{\nu_{1},\nu_{2},\nu_{3}\}$ of $N_{p}M$, the curvature parabola $\Delta_{p}$ can be parametrised by
\begin{equation}\label{parabola}
    \eta(y)=\sum_{i=1}^{3}(l_{\nu_{i}}+2m_{\nu_{i}}y+n_{\nu_{i}}y^{2})\nu_{i}.
\end{equation}
Besides, $\Delta_{p}$ is a plane curve that may degenerate.

The set $J^{2}(2,4)$ denotes the space of $2$-jets $j^{2}f(0)$ of map germs $f:(\mathbb{R}^{2},0)\rightarrow(\mathbb{R}^{4},0)$ and $\Sigma^{1}J^{2}(2,4)$ will stand for the subset of corank $1$ $2$-jets. Also, $\mathcal{A}^{2}$ denotes the space of $2$-jets of diffeomorphisms in the source and target. It is shown in \cite{Mendes} that there are four $\mathcal{A}^{2}$ orbits in $\Sigma^{1}J^{2}(2,4)$. For completeness we give a sketch of the proof.

\begin{lem}[\cite{Mendes}]\label{lema2jato}
There exists four $\mathcal{A}^{2}$ orbits in $\Sigma^{1}J^{2}(2,4)$:
$$(x,xy,y^{2},0),\ (x,y^{2},0,0),\ (x,xy,0,0)\ \mbox{and}\ (x,0,0,0).$$
\end{lem}
\begin{proof}
Consider the $2$-jet
$$j^{2}f(0)=(x,a_{20}x^{2}+a_{11}xy+a_{02}y^{2},b_{20}x^{2}+b_{11}xy+b_{02}y^{2},c_{20}x^{2}+c_{11}xy+c_{02}y^{2}).$$
A change of coordinates in the target gives us
$$j^{2}f(0)\sim(x,a_{11}xy+a_{02}y^{2},b_{11}xy+b_{02}y^{2},c_{11}xy+c_{02}y^{2}).$$
Also, a change of coordinates in the source
leads us to
$$j^{2}f(0)\sim\left(x,\frac{a_{11}b_{02}-a_{02}b_{11}}{b_{02}}xy,\frac{a_{11}b_{02}-a_{02}b_{11}}{a_{11}}y^{2},c_{11}xy+c_{02}y^{2}\right).$$
The change of coordinates could have been made using any two of the last three components. The result would be similar, so would the conditions. Without loss of generality, we shall use the second and the third ones.
Hence, if $a_{11}b_{02}-a_{02}b_{11}\neq0$ (or $a_{11}c_{02}-a_{02}c_{11}\neq0$ or
$c_{11}b_{02}-c_{02}b_{11}\neq0$, in the other cases), then, using just changes in the target, we have
$j^{2}f(0)\sim(x,xy,y^{2},c_{11}xy+c_{02}y^{2})$. Finally, we have $j^{2}f(0)\sim(x,xy,y^{2},0)$.
On the other hand, if $a_{11}b_{02}-a_{02}b_{11}=0$, the coordinates $Y$ and $Z$ are linearly dependent, that means there are real numbers $A,B\neq0$ such that $AY+BZ=0$. Taking $\widetilde{X}=X,\ \widetilde{Y}=Y,\ \widetilde{Z}=AY+BZ\ \mbox{and}\ \widetilde{W}=W$
and making a rotation in the target, we get
$j^{2}f(0)\sim(x,a_{11}xy+a_{02}y^{2},\tilde{c}_{11}xy+\tilde{c}_{02}y^{2},0)$.
Continuing the analysis, if $a_{11}\tilde{c}_{02}-a_{02}\tilde{c}_{11}\neq0$,
with a change of coordinates in the target, as done before, $j^{2}f(0)\sim(x,xy,y^{2},0)$. However, if $a_{11}\tilde{c}_{02}-a_{02}\tilde{c}_{11}=0$, we may take $\tilde{c}_{02}=\tilde{c}_{11}=0$
and then, $j^{2}f(0)\sim(x,a_{11}xy+a_{02}y^{2},0,0)$. We have the following possibilities:
\begin{description}
\item[(i)] If $a_{02}\neq0$ and $a_{11}\in\mathbb{R}$, $j^{2}f(0)\sim(x,y^{2},0,0)$.
\item[(ii)] If  $a_{02}=0$ and $a_{11}\neq0$, it follows directly $j^{2}f(0)\sim(x,xy,0,0)$.
\item[(iii)] If $a_{02}=a_{11}=0$, $j^{2}f(0)\sim(x,0,0,0)$.
\end{description}
Thus, we have the four $\mathcal{A}^{2}$ orbits in $\Sigma^{1}J^{2}(2,4)$.
\end{proof}

\begin{rem}\label{obs.coef.}
Given a $2$-jet,
$$j^{2}f(0)=(x,a_{20}x^{2}+a_{11}xy+a_{02}y^{2},b_{20}x^{2}+b_{11}xy+b_{02}y^{2},c_{20}x^{2}+c_{11}xy+c_{02}y^{2}),$$
of a corank $1$ map germ, we give in  Table \ref{condicoes} a criterion over its coefficients in order to identify to which of the previous four orbits given by Lemma \ref{lema2jato} it belongs.
\begin{table}[h]
\caption{Conditions over the coefficients of the $2$-jet for the $\mathcal{A}^{2}$-classification}
\centering
{\begin{tabular}{ccc}
\hline
$\mathcal{A}^{2}$-normal form & Conditions\\
\hline
$(x,xy,y^{2},0)$ & $\gamma_{1}=a_{11}b_{02}-a_{02}b_{11}\neq0\ \mbox{or}\ \gamma_{2}=a_{11}c_{02}-a_{02}c_{11}\neq0$ \\
               & $\mbox{or}\ \gamma_{3}=c_{11}b_{02}-c_{02}b_{11}\neq0$\\ \cr
$(x,y^2,0,0)$ & $\gamma_{1}=\gamma_{2}=\gamma_{3}=0\ \mbox{and}\ a_{02}^{2}+b_{02}^{2}+c_{02}^{2}>0$ \\ \cr
$(x,xy,0,0)$ & $\gamma_{1}=\gamma_{2}=\gamma_{3}=0,\ a_{02}=b_{02}=c_{02}=0\ \mbox{and}\ a_{11}^{2}+b_{11}^{2}+c_{11}^{2}>0$\\ \cr
$(x,0,0,0)$ & $a_{11}=a_{02}=b_{11}=b_{02}=c_{11}=c_{02}=0.$\cr
\hline
\end{tabular}
}
\label{condicoes}
\end{table}

\end{rem}

The next result shows that the curvature parabola $\Delta_{p}$ is a complete invariant for the previous classification (Lemma \ref{lema2jato}), in the sense that it distinguishes the four $\mathcal{A}^{2}$-orbits.

\begin{teo}\label{teo.2jato/parabola}
Let $M\subset\mathbb{R}^{4}$ be a corank $1$ surface at $p$. We shall assume $p=(0,0,0,0)\in M$ and denote by $j^{2}f(0)$ the $2$-jet of a local parametrisation $f:(\mathbb{R}^{2},0)\rightarrow(\mathbb{R}^{4},0)$ of $M$. The following equivalences hold:
\begin{itemize}
    \item[(a)] $\Delta_{p}$ is a non-degenerate parabola iff $j^{2}f(0)\sim_{\mathcal{A}^{2}} (x,xy,y^{2},0)$;
    \item[(b)] $\Delta_{p}$ is a half-line iff $j^{2}f(0)\sim_{\mathcal{A}^{2}} (x,y^{2},0,0)$;
    \item[(c)] $\Delta_{p}$ is a line iff $j^{2}f(0)\sim_{\mathcal{A}^{2}} (x,xy,0,0)$;
    \item[(d)] $\Delta_{p}$ is a point iff $j^{2}f(0)\sim_{\mathcal{A}^{2}} (x,0,0,0)$.
\end{itemize}
\end{teo}
\begin{proof}
We shall assume without loss of generality that
$$j^{2}f(0)=\left(x,\frac{1}{2}(a_{20}x^{2}+2a_{11}xy+a_{02}y^{2}),\frac{1}{2}(b_{20}x^{2}+2b_{11}xy+b_{02}y^{2}),
\frac{1}{2}(c_{20}x^{2}+2c_{11}xy+c_{02}y^{2})\right)$$
and $\{\textbf{e}_{1},\textbf{e}_{2},\textbf{e}_{3},\textbf{e}_{4}\}$ will denote an orthonormal frame of $\mathbb{R}^{4}$. Thus,
$T_{p}M=[\textbf{e}_{1}]$ and $N_{p}M=[\textbf{e}_{2},\textbf{e}_{3},\textbf{e}_{4}]$.
The matrix of coefficients of the second fundamental form is given by:
$$\left(
            \begin{array}{ccc}
              a_{20} & a_{11} & a_{02} \\
              b_{20} & b_{11} & b_{02} \\
              c_{20} & c_{11} & c_{02} \\
            \end{array}
          \right).
$$
Even though $\Delta_{p}$ is a curve whose trace lies in $\mathbb{R}^{4}$, we shall consider it a curve in a three-dimensional space and parametrise it by
$$\eta(y)=(0,a_{20}+2a_{11}y+a_{02}y^{2},b_{20}+2b_{11}y+b_{02}y^{2},c_{20}+2c_{11}y+c_{02}y^{2}).$$
The curvature of $\eta$ is
$$\kappa(y)=\frac{|\eta'(y)\times\eta''(y)|}{|\eta'(y)|^{3}},$$
and it will be identically zero iff $\eta'(y)\times\eta''(y)=0$.
Calculations show us that
$$\eta'(y)\times\eta''(y)=4(c_{02}b_{11}-b_{02}c_{11},a_{02}c_{11}-c_{02}a_{11},b_{02}a_{11}-a_{02}b_{11}).$$
Therefore, the curvature parabola $\Delta_{p}$ is a degenerate parabola if, and only if,
$$c_{02}b_{11}-b_{02}c_{11}=a_{02}c_{11}-c_{02}a_{11}=b_{02}a_{11}-a_{02}b_{11}=0,$$
which is equivalent to $\gamma_{1}=\gamma_{2}=\gamma_{3}=0$ due to Remark \ref{obs.coef.}. The other cases follow from the parametrisation of $\Delta_{p}$ and the Remark \ref{obs.coef.}.
\end{proof}

The group of $2$-jets of diffeomorphisms from $(\mathbb{R}^{2},0)$ to $(\mathbb{R}^{2},0)$ is denoted by $\mathcal{R}^{2}$ and $\mathcal{O}(4)$ is the group of linear isometries in $\mathbb{R}^{4}$. Thus $\mathcal{R}^{2}\times\mathcal{O}(4)$ is a subgroup of $\mathcal{A}^{2}$ that also acts on $\Sigma^{1}J^{2}(2,4)$. Here, $\mathcal{M}_{2}$ denotes the maximal ideal in the local ring of smooth function germs $(\mathbb{R}^{2},0)\rightarrow(\mathbb{R},0)$ and for each $k\geq1$, $\mathcal{M}_{2}^{k}$ is the ideal of functions whose $(k-1)$-jet is zero.

In order to show that two $2$-jets of local parametrisations are equivalent under the action of $\mathcal{R}^{2}\times\mathcal{O}(4)$ if and only if there exists an isometry preserving the respective curvature parabola, we shall need some ``special parametrisations", which are given in the next lemma. These parametrisations are the ones obtained by smooth changes of coordinates in the source and isometries in the target.

\begin{lem}\label{lemma2jato2}
Let $f:(\mathbb{R}^{2},0)\rightarrow(\mathbb{R}^{4},0)$ be a corank $1$ map germ. Using only smooth changes of coordinates in the source and isometries in the target, $f$ can be reduced to one of the following forms:
\begin{itemize}
\item[(i)] $(x,y)\mapsto(x,xy+p(x,y),b_{20}x^{2}+b_{11}xy+b_{02}y^{2}+q(x,y),c_{20}x^{2}+r(x,y))$, if
$j^{2}f(0)\sim_{\mathcal{A}^{2}}(x,xy,y^{2},0)$;
\item[(ii)] $(x,y)\mapsto(x,a_{20}x^{2}+y^{2}+p(x,y),b_{20}x^{2}+q(x,y),r(x,y))$, if
$j^{2}f(0)\sim_{\mathcal{A}^{2}}(x,y^{2},0,0)$;
\item[(iii)] $(x,y)\mapsto(x,xy+p(x,y),b_{20}x^{2}+q(x,y),r(x,y))$, if $j^{2}f(0)\sim_{\mathcal{A}^{2}}(x,xy,0,0)$;
\item[(iv)] $(x,y)\mapsto(x,p(x,y),b_{20}x^{2}+q(x,y),r(x,y))$, if $j^{2}f(0)\sim_{\mathcal{A}^{2}}(x,0,0,0)$,
\end{itemize}
where $a_{ij},b_{ij},c_{ij}\in\mathbb{R}$, $b_{02}>0$ and $p,q,r\in\mathcal{M}_{2}^{3}$.
\end{lem}
\begin{proof}
The proof will be split into cases.
In the first one, consider $f$ a corank $1$ map germ such that $j^{2}f(0)\sim(x,xy,y^{2},0)$.
We may write $f$ as:
$$(x,\frac{a_{20}}{2}x^{2}+a_{11}xy+\frac{a_{02}}{2}y^{2}+o(3),\frac{b_{20}}{2}x^{2}+b_{11}xy+\frac{b_{02}}{2}y^{2}+o(3),
\frac{c_{20}}{2}x^{2}+c_{11}xy+\frac{c_{02}}{2}y^{2}+o(3)),$$
where $\gamma_{1}=a_{11}b_{02}-a_{02}b_{11}\neq0$ or $\gamma_{2}=a_{11}c_{02}-a_{02}c_{11}\neq0$ or $\gamma_{3}=c_{11}b_{02}-c_{02}b_{11}\neq0$ (see Remark \ref{obs.coef.}).
Suppose, without loss of generality that $\gamma_{1}\neq0$. Consider the angle $\theta_{1}$
satisfying
$$(\sin\theta_{1},\cos\theta_{1})=\frac{(a_{02},b_{02})}{\sqrt{a_{02}^{2}+b_{02}^{2}}}.$$
Also, consider the rotation in $\mathbb{R}^{4}$ given by
$$T_{1}=\left(
      \begin{array}{cccc}
        1 & 0 & 0 & 0 \\
        0 & \cos\theta_{1} & -\sin\theta_{1} & 0 \\
        0 & \sin\theta_{1} & \cos\theta_{1} & 0 \\
        0 & 0 & 0 & 1 \\
      \end{array}
    \right)
$$
and the change of coordinates in the source
$$\psi(x,y)=\left(x,\frac{1}{\gamma_{1}}(-\frac{1}{2}(a_{20}b_{02}-a_{02}b_{20})x+\sqrt{a_{02}^{2}+b_{02}^{2}}y)\right).$$
Hence,
$$(T_{1}\circ g\circ\psi)(x,y)=\left(x,xy+\sum_{i+j=3}\frac{\lambda_{ij}}{i!j!}x^{i}y^{j}+o(4),\sum_{i+j=2}\frac{\mu_{ij}}{i!j!}x^{i}y^{j}+o(3),
\sum_{i+j=2}\frac{\sigma_{ij}}{i!j!}x^{i}y^{j}+o(3)\right),$$
for constants $\lambda_{ij}$, $\mu_{ij}$ and $\sigma_{ij}$ and
rewriting,
$$f\sim(x,xy+p(x,y),a'x^{2}+b'xy+c'y^{2}+q'(x,y),a''x^{2}+b''xy+c''y^{2}+r'(x,y)),$$
where $p,q',r'\in\mathcal{M}_{2}^{3}$.
A rotation in the axes $ZW$ eliminates $c''y^{2}$ and
$$f\sim(x,xy+p(x,y),ax^{2}+bxy+cy^{2}+q(x,y),\tilde{a}x^{2}+\tilde{b}xy+r(x,y)),$$
where $p,q,r\in\mathcal{M}_{2}^{3}$.
Using the rotation $T_{2}$ given by
$$T_{2}=\left(
      \begin{array}{cccc}
        1 & 0 & 0 & 0 \\
        0 & \sin\theta_{2} & 0 & \cos\theta_{2} \\
        0 & 0 & 1 & 0 \\
        0 & -\cos\theta_{2} & 0 &\sin\theta_{2} \\
      \end{array}
    \right)
$$
of the angle $\theta_{2}$ such that $(\sin\theta_{2},\cos\theta_{2})=\frac{(1,\tilde{b})}{\sqrt{1+\tilde{b}^{2}}}$
and the change of coordinates in the source  $\bar{\varphi}(x,y)=(x,1/(\sin\theta_{2}+\tilde{b}\cos\theta_{2})y)$, we obtain
$$f\sim(x,xy+\tilde{a}\cos\theta_{2}x^{2}+\bar{p}(x,y),ax^{2}+bxy+cy^{2}+q(x,y),\bar{a}x^{2}+\bar{r}(x,y)).$$
The last change is also made in the source: $\tilde{\varphi}(x,y)=(x,y-\tilde{a}\cos\theta_{2}x)$. Therefore the first normal form is
$$(x,y)\mapsto(x,xy+p(x,y),b_{02}x^{2}+b_{11}xy+b_{02}y^{2}+q(x,y),c_{02}x^{2}+r(x,y)).$$
We shall assume $b_{02}>0$, since it is non-zero. Indeed, if $b_{02}<0$ we can rotate the axes $ZW$ by $\theta_{3}=\pi$. The remaining cases are very similar and their proof will be omitted.
\end{proof}

The following theorem is the main result of this section. It shows that the curvature parabola codifies the second order information of a corank $1$ surface $M\subset\mathbb{R}^{4}$. This means that the curvature parabola $\Delta_{p}$ and its position in the normal space $N_{p}M$ contains all the second order geometry of the surface at the singular point $p$.

\begin{teo}\label{mainteo}
Let $M_{1},M_{2}\subset\mathbb{R}^{4}$ be two corank $1$ surfaces
at $p_{1}\in M_{1}$ and $p_{2}\in M_{2}$ and parametrised by $f$ and $g$, respectively. The
$2$-jets $j^{2}f(0)$ and $j^{2}g(0)\in \Sigma^{1}J^{2}(2,4)$ are equivalent under the action of
$\mathcal{R}^{2}\times \mathcal{O}(4)$ if and only if, there is a linear isometry
$\varphi:N_{p_{1}}M_{1}\rightarrow N_{p_{2}}M_{2}$ such that $\varphi(\Delta_{p_{1}}M_{1})=\Delta_{p_{2}}M_{2}$.
\end{teo}
\begin{proof} Assume that $j^{2}f(0)$ and $j^{2}g(0)$ are $\mathcal{A}^{2}$-equivalent to $(x,xy,y^{2},0)$.
By Lemma \ref{lemma2jato2}, we can reduce them to
$$j^{2}f(0)=(x,xy,ax^{2}+bxy+cy^{2},dx^{2})\ \mbox{and}\ j^{2}g(0)=(x,xy,\tilde{a}x^{2}+\tilde{b}xy+\tilde{c}y^{2},\tilde{d}x^{2}).$$
We shall denote the coordinates in $\mathbb{R}^{4}$ by $(X,Y,Z,W)$. Both normal hyperplanes $N_{p_{1}}M_{1}$ and $N_{p_{2}}M_{2}$ coincide with the hyperplane $YZW$. The curvature parabolas $\Delta_{p_{1}}M_{1}$ and $\Delta_{p_{2}}M_{2}$ can be respectively parametrised by:
$$\eta_{1}(y)=(0,2y,2a+2by+2cy^{2},2d)\ \mbox{and}\ \eta_{2}(y)=(0,2y,2\tilde{a}+2\tilde{b}y+2\tilde{c}y^{2},2\tilde{d}).$$
Suppose that $j^{2}f(0)\sim j^{2}g(0)$ under the action of $\mathcal{R}^{2}\times \mathcal{O}(4)$. Thus, there is a pair
$(\psi,\varphi)\in\mathcal{R}^{2}\times \mathcal{O}(4)$ such that $\varphi\circ j^{2}f(0)\circ\psi=j^{2}g(0)$.
We can assume $\psi$ linear, since both $2$-jets are homogeneous. The matrices of $\psi$ and $\varphi$ are, respectively:
$$P=\left(
    \begin{array}{cc}
      p & q \\
      r & s \\
    \end{array}
  \right)\ \mbox{and}\
 A= \left(
    \begin{array}{cccc}
      a_{11} & a_{12} & a_{13} & a_{14} \\
      a_{21} & a_{22} & a_{23} & a_{24} \\
      a_{31} & a_{32} & a_{33} & a_{34} \\
      a_{41} & a_{42} & a_{43} & a_{44} \\
    \end{array}
  \right)$$

\noindent satisfying $\det(P)\neq0$ and $AA^{t}=I$.
We write $\varphi\circ j^{2}f(0)\circ\psi=(h_{1},h_{2},h_{3},h_{4})$, with
$$\begin{array}{cl}
    h_{i} & =a_{i1}px+a_{i1}qy+a_{i2}(prx^{2}+(ps+qr)xy+qsy^{2})+a_{i3}(ap^{2}+bpr+cr^{2})x^{2} \\
          & +a_{i3}(2apq+b(ps+qr)+2crs)xy+a_{i3}(aq^{2}+bqs+cs^{2})y^{2}+a_{i4}dp^{2}x^{2} \\
          & +2a_{i4}dpqxy+a_{i4}dq^{2}y^{2},\ \ i=1,2,3,4.
  \end{array}
$$
Comparing the terms of $\varphi\circ j^{2}f(0)\circ\psi$ and $j^{2}g(0)$:
$$\begin{array}{l}
    a_{11}p=1\ \ a_{12}p=a_{13}p=a_{14}p=0 \\
    a_{11}q=a_{12}q=a_{13}q=a_{14}q=0 \\
    a_{12}pr+a_{13}(ap^{2}+bpr+cr^{2})+a_{14}dp^{2}=0 \\
    a_{22}pr+a_{23}(ap^{2}+bpr+cr^{2})+a_{24}dp^{2}=0 \\
    a_{32}pr+a_{33}(ap^{2}+bpr+cr^{2})+a_{34}dp^{2}=\tilde{a} \\
    a_{42}pr+a_{43}(ap^{2}+bpr+cr^{2})+a_{44}dp^{2}=\tilde{d} \\
    a_{12}(ps+qr)+a_{13}(2apq+bps+bqr+2crs)+2a_{14}dpq=0 \\
    a_{22}(ps+qr)+a_{23}(2apq+bps+bqr+2crs)+2a_{24}dpq=1 \\
    a_{32}(ps+qr)+a_{33}(2apq+bps+bqr+2crs)+2a_{34}dpq=\tilde{b} \\
    a_{42}(ps+qr)+a_{43}(2apq+bps+bqr+2crs)+2a_{44}dpq=0 \\
    a_{12}qs+a_{13}(aq^{2}+bqs+cs^{2})+a_{14}dq^{2}=0 \\
    a_{22}qs+a_{23}(aq^{2}+bqs+cs^{2})+a_{24}dq^{2}=0 \\
    a_{32}qs+a_{33}(aq^{2}+bqs+cs^{2})+a_{34}dq^{2}=\tilde{c} \\
    a_{42}qs+a_{43}(aq^{2}+bqs+cs^{2})+a_{44}dq^{2}=0.
  \end{array}
$$
From the equations above and using the conditions: $\det(P)=ps-qr\neq0$ and $AA^{t}=I_{4}$,
we have $a_{11},p,s,a_{33}\neq0$ and $r=q=a_{12}=a_{13}=a_{14}=a_{21}=a_{31}=a_{41}=a_{23}=a_{43}
=a_{32}=a_{34}=a_{24}=a_{42}=0$.
Hence,
$$\begin{array}{l}
    a_{11}p=p^{2}=s^{2}=1,\ \tilde{a}=a_{33}a \\
    \tilde{d}=a_{44}d,\ a_{22}ps=1 \\
    \tilde{b}=a_{33}bps,\ \tilde{c}=a_{33}c \\
    a_{11}^{2}=a_{22}^{2}=a_{33}^{2}=a_{44}^{2}=1.
  \end{array}
$$

The isometry we are looking for is precisely $\varphi$.
To show it, it is enough to prove that $\varphi\circ\eta_{1}$ and $\eta_{2}$ have the same trace, in other words, there is a change of coordinates that maps one parametrisation to another.

Since $a_{22}$ and $ps$ have the same sign, or both parametrisations coincide or we just need to consider the change $y\mapsto -y$. For the converse, suppose that there is a linear isometry $\varphi:N_{p_{1}}M_{1}\rightarrow N_{p_{2}}M_{2}$ satisfying $\varphi(\Delta_{p_{1}}M_{1})=\Delta_{p_{2}}M_{2}$. The matrix of $\varphi$ is:
$$B=\left(
            \begin{array}{ccc}
              b_{11} & b_{12} & b_{13} \\
              b_{21} & b_{22} & b_{23} \\
              b_{31} & b_{32} & b_{33} \\
            \end{array}
          \right).
$$
Comparing $\varphi\circ\eta_{1}$ and $\eta_{2}$, we obtain:
$$
\begin{array}{l}
b_{12}=b_{13}=b_{31}=b_{32}=0,\ b_{11}=1 \\
\tilde{c}=b_{22}c \\
\tilde{b}=b_{21}+b_{22}b \\
\tilde{a}=b_{22}a+b_{23}d \\
\tilde{d}=b_{33}d.
\end{array}
$$

Since $\varphi$ is a linear isometry, $B$ is orthogonal: $BB^{t}=I_{3}$. Adding this condition, we have $b_{21}=b_{23}=0$,
$b_{22}^{2}=b_{33}^{2}=1$ and $b_{22}b_{33}=1$.

Thus,
$$\tilde{a}=\pm a,\ \tilde{b}=\pm b,\ \tilde{c}=\pm c,\mbox{e}\ \tilde{d}=\pm d,$$
where the signs coincide since the same happens to $b_{22}$ and $b_{33}$.
Lastly, we can extend $\varphi$ to a linear isometry of $\mathbb{R}^{4}$, and then
$j^{2}f(0)\sim j^{2}g(0)$ under the action of $\mathcal{R}^{2}\times\mathcal{O}(4)$.
The remaining cases proceed in a similar way.
\end{proof}


\section{Second order properties}
In this section we aim to define and study some second order
invariants (that is, properties that are invariant under the action
of $\mathcal{R}^{2}\times\mathcal{O}(4)$) for corank $1$ surfaces in
$\mathbb{R}^{4}$. We shall look at the geometry of regular surfaces
in $\mathbb{R}^{4}$. This study is local and the procedure is as
follows. Given a regular surface $N\subset\mathbb{R}^{5}$, we
consider the corank $1$ surface $M$ at $p$ obtained by the
projection of $N$ in a tangent direction, via the map
$\xi:N\subset\mathbb{R}^{5}\rightarrow M$. The regular surface
$N\subset\mathbb{R}^{5}$ can be taken, locally, as the image of an
immersion $i:\tilde{M}\rightarrow N\subset\mathbb{R}^{5}$, where
$\tilde{M}$ is the regular surface from the construction done
before.

The local correspondence between these two surfaces carries out an
important geometric consequence: the regular surface has the same
second fundamental form as $M$. The points of $N$ can be
characterized according to the rank of its fundamental form at that
point. Inspired by this classification, we can classify the singular
points of the surface $M$:

\begin{definition}
Given a singular point of corank $1$, $p\in M$, we say that $p\in M_{i}$ if the rank of the second fundamental form at $p$ is $i$, $i=0,1,2,3$.
\end{definition}


\begin{definition}
The minimal affine space which contains the curvature parabola is denoted by $\mathcal{A}ff_{p}$. The plane denoted by $E_{p}$ is the vector space: parallel to $\mathcal{A}ff_{p}$ when $\Delta_{p}$ is a non degenerate parabola, the plane through $p$ that contains $\mathcal{A}ff_{p}$ when $\Delta_{p}$ is a non radial half-line or a non radial line and any plane through $p$ that contains $\mathcal{A}ff_{p}$ when $\Delta_{p}$ is a radial half-line, a radial line or a point.
\end{definition}

\begin{lem}\label{lemma/tipo.de.pontos}
Let $p$ be a singular point of corank $1$ of the surface $M\subset\mathbb{R}^{4}$. Then, the following holds:
\begin{enumerate}
    \item[(i)] When $\Delta_{p}$ is a non degenerate parabola, $p\in M_{2}$ or $p\in M_{3}$ according to $\mathcal{A}ff_{p}=E_{p}$ or not, respectively;
    \item[(ii)] When $\Delta_{p}$ is a half-line or a line, $p\in M_{1}$ or $p\in M_{2}$ depending on $\Delta_{p}$ being radial or not, respectively;
    \item[(iii)] When $\Delta_{p}$ is a point, $p\in M_{0}$ or $p\in M_{1}$ according to $\Delta_{p}$ is $p$ or not, respectively.
\end{enumerate}
\end{lem}
\begin{proof}
Using the normal forms from Lemma \ref{lemma2jato2}, the proof follows from straightforward calculations.
\end{proof}


\begin{figure}[h!]
\begin{center}
\includegraphics[scale=0.45]{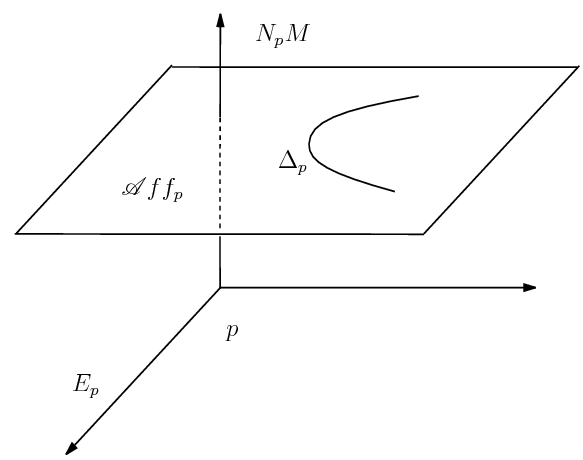}
\caption{The planes $\mathcal{A}ff_{p}$ and $E_{p}$}
\label{planes}
\end{center}
\end{figure}

Let $S\subset\mathbb{R}^{4}$ be the regular surface locally obtained by projecting $N\subset\mathbb{R}^{5}$ via the map $\texttt{p}$ into the four space given by $T_{\xi^{-1}(p)}N\oplus \xi^{-1}(E_{p})$ (see the following diagram).


$$
\xymatrix{
 &  & N\subset\mathbb{R}^{5}\ar[rd]^-{\texttt{p}}\ar[d]^-{\xi} & \\
 \mathbb{R}^{2}\ar@/_0.7cm/[rr]^-{f}& \tilde{M}\ar[l]_-{\phi}\ar[r]^-{g}\ar[ru]^-{i}& M\subset\mathbb{R}^{4}& S\subset\mathbb{R}^{4}
}
$$

\begin{ex}\label{ex.construcao}
Let $M\subset\mathbb{R}^{4}$ a corank $1$ surface at $p\in M$ given by $f(x,y)=(x,f_{2}(x,y),f_{3}(x,y),f_{4}(x,y))$, with $f_{i}\in\mathcal{M}_{2}^{2}$, for $i=2,3,4$. Notice that from Lemma \ref{lemma2jato2} we can always take $M$ locally parametrised as above and for any normal form, $E_{p}$ is the plane $YZ$ in $\mathbb{R}^{4}$. Also, let $N\subset\mathbb{R}^{5}$ be a regular surface given by $i(x,y)=(x,y,f_{2}(x,y),f_{3}(x,y),f_{4}(x,y))$, so that $M$ is the projection of $N$ in the tangent direction $(0,1)$. Then, the regular surface $S\subset\mathbb{R}^{4}$ is locally given by $(x,y)\mapsto(x,y,f_{2}(x,y),f_{3}(x,y))$.
\end{ex}

\begin{prop}
Suppose $p\in M_{i}$, $i=0,1$. The second order geometry of the surface $S\subset\mathbb{R}^{4}$ obtained by the previous construction does not depend on the choice of the plane $E_{p}\subset N_{p}M$.
\end{prop}
\begin{proof}
When $p\in M_{1}$, that is, $\Delta_{p}$ is a radial half-line, a radial line or a point other than $p$, we can assume that $\Delta_{p}$ is contained in one of the coordinate axes, for instance, the $Y$-axis. Let $\Gamma_{1},\Gamma_{2}$ be two distinct planes through $p$ such that $\Delta_{p}\subset\Gamma_{1},\Gamma_{2}$. Consider the regular surface $N\subset\mathbb{R}^{5}$, as before and the projection $\xi:N\rightarrow M$. The regular surfaces $S_{1},S_{2}\subset\mathbb{R}^{4}$ obtained by the projection of $N$ into the four spaces $T_{\xi^{1}(p)}N\oplus\xi^{-1}(\Gamma_{1})$ and $T_{\xi^{1}(p)}N\oplus\xi^{-1}(\Gamma_{2})$, respectively, are not the same, however, they have the same second fundamental form. Taking the local parametrisation of $M$ at $p$, $f(x,y)=(x,f_{2}(x,y),f_{3}(x,y),f_{4}(x,y))$, $f_{2}\in\mathcal{M}_{2}^{2}$, $f_{3},f_{4}\in\mathcal{M}_{2}^{3}$, we can locally parametrise $S_{1}$ and $S_{2}$ by
$$
g_{1}(x,y)=(x,y,f_{2}(x,y),\alpha_{1}f_{2}(x,y)+\beta_{1}f_{3}(x,y)+\gamma_{1}f_{4}(x,y))$$ and $$g_{2}(x,y)=(x,y,f_{2}(x,y),\alpha_{2}f_{2}(x,y)+\beta_{2}f_{3}(x,y)+\gamma_{2}f_{4}(x,y)),
$$
respectively, for $\alpha_{i},\beta_{i},\gamma_{i}\in\mathbb{R}$, $i=1,2$.
Rotations on the target take $g_{1}$ and $g_{2}$ to
$$\tilde{g}_{1}(x,y)=(x,y,f_{2}(x,y)+p_{1}(x,y),q_{1}(x,y))\ \mbox{and}\ \tilde{g}_{2}(x,y)=(x,y,f_{2}(x,y)+p_{2}(x,y),q_{2}(x,y)),$$
respectively, where $p_{i},q_{i}\in\mathcal{M}_{2}^{3}$, for $i=1,2$. Hence, their second fundamental forms are equivalent up to rotations on the target. The case $p\in M_{0}$ is analogous.
\end{proof}


In this section we show the relation of the corank $1$ singular surface $M\subset\mathbb{R}^{4}$ and the regular surface $S\subset\mathbb{R}^{4}$.

\subsection{Asymptotic directions}

The definitions of aymptotic and binormal directions are inspired by both definitions of such directions of regular surfaces in $\mathbb{R}^{4}$ and by corank $1$ surfaces in $\mathbb{R}^{3}$. These objects are defined using the second fundamental form of the surface and later will be characterized in terms of the curvature parabola.

\begin{definition}
A non zero direction $\textbf{u}\in T_{q}\tilde{M}$ is called \emph{asymptotic} if there is a non zero vector $\nu\in E_{p}$ such that
$$II_{\nu}(\textbf{u},\textbf{v})=\langle II(\textbf{u},\textbf{v}),\nu\rangle=0\ \ \forall\ \textbf{v}\in T_{q}\tilde{M}.$$
Moreover, in such case, we say that $\nu$ is a \emph{binormal direction}.
\end{definition}

The normal vectors $\nu\in N_{p}M$ satisfying the condition
$II_{\nu}(\textbf{u},\textbf{v})=0$ are called \emph{degenerate
directions}, but only those in $E_p$ are binormal directions.

\begin{lem}
When $p\in M_{1}\cup M_{0}$, the choice of $E_{p}$ does not change the number of binormal directions. Furthermore, all directions $\textbf{u}\in T_{q}\tilde{M}$ are asymptotic.
\end{lem}
\begin{proof}
 Indeed, if $\Delta_{p}$ is a radial half-line or radial line, the degenerate directions are the ones of the plane orthogonal to $\mathcal{A}ff_{p}$. Choosing any plane from the pencil of planes containing $\Delta_{p}$, the intersection of this plane with the one orthogonal to $\mathcal{A}ff_{p}$ gives us one binormal direction. When $\Delta_{p}$ is a point other than $p$, the situation is the same: we just have to take the line through $p$ and $\Delta_{p}$ and the consideration is the same as above. However, when $\Delta_{p}=\{p\}$ all directions of $E_{p}$ are binormals for any choice of $E_{p}$. For the second part it is enough to take $\nu\in N_{p}M$ such that $\nu$ is orthogonal to $\Delta_{p}$, when it is a half-line or a line, or $\nu$ orthogonal to the line containing $p$ and $\Delta_{p}$, when it is a point.
\end{proof}

\begin{rem}\label{obs-quebra-do-normal} Consider $\{\nu_{1},\nu_{2},\nu_{3}\}$ an orthonormal frame of $N_{p}M$ such that $E_{p}=[\nu_{1},\nu_{2}]$ and $E_{p}^{\perp}=[\nu_{3}]$. Hence, we may write $N_{p}M=E_{p}\oplus E_{p}^{\perp}$ and the matrix of the second fundamental form of $M$ will be given by
$$
\left(
  \begin{array}{ccc}
    l_{\nu_{1}} & m_{\nu_{1}} & n_{\nu_{1}} \\
    l_{\nu_{2}} & m_{\nu_{2}} & n_{\nu_{2}} \\
    l_{\nu_{3}} & m_{\nu_{3}} & n_{\nu_{3}} \\
  \end{array}
\right),
$$
where $l_{\nu_{i}}$, $m_{\nu_{i}}$ and $n_{\nu_{i}}$, $i=1,2,3$ are the coefficients corresponding to the orthonormal frame $\{\nu_{1},\nu_{2},\nu_{3}\}$.
\end{rem}

The next result gives a criterion to determine whether a tangent direction $\textbf{u}\in T_{q}\tilde{M}$ is asymptotic or not.

\begin{lem}\label{lema-l.a.}
Let $E_{p}=[\nu_{1},\nu_{2}]$ be as in Remark \ref{obs-quebra-do-normal}.
A non zero tangent direction $\textbf{u}=(\alpha,\beta)\in T_{q}\tilde{M}$ is asymptotic if and only if
$$\left|
    \begin{array}{ccc}
      \beta^{2} & -\alpha\beta & \alpha^{2} \\
      l_{\nu_{1}} & m_{\nu_{1}} & n_{\nu_{1}} \\
      l_{\nu_{2}} & m_{\nu_{2}} & n_{\nu_{2}} \\
    \end{array}
  \right|=0.
$$
\end{lem}
\begin{proof}
Consider a tangent direction $\textbf{v}=(\bar{\alpha},\bar{\beta})\in T_{q}\tilde{M}$ and a non zero normal vector
$\nu=a\nu_{1}+b\nu_{2}\in E_{p}$. There is no need of a component in the $\nu_{3}$ direction, since $\nu\in E_{p}$ (see Remark \ref{obs-quebra-do-normal}). Thus,
$$\begin{array}{ll}
II(\textbf{u},\textbf{v}) & =II(\alpha\partial_{x}+\beta\partial_{y},\bar{\alpha}\partial_{x}+\bar{\beta}\partial_{y}) \\
        & =\alpha\bar{\alpha}II(\partial_{x},\partial_{x})+(\alpha\bar{\beta}+\beta\bar{\alpha})II(\partial_{x},\partial_{y})+\beta\bar{\beta}II(\partial_{y},\partial_{y})
\end{array}
$$
and
$$\begin{array}{ll}
II_{\nu}(\textbf{u},\textbf{v}) & =\langle II(\textbf{u},\textbf{v}),a\nu_{1}+b\nu_{2}\rangle \\
              & =a\langle II(\textbf{u},\textbf{v}),\nu_{1}\rangle+b\langle II(\textbf{u},\textbf{v}),\nu_{2}\rangle \\
              & =a[\alpha\bar{\alpha}l_{\nu_{1}}+(\alpha\bar{\beta}+\beta\bar{\alpha})m_{\nu_{1}}+\beta\bar{\beta}n_{\nu_{1}}] \\ &\ +b[\alpha\bar{\alpha}l_{\nu_{2}}+(\alpha\bar{\beta}+\beta\bar{\alpha})m_{\nu_{2}}+\beta\bar{\beta}n_{\nu_{2}}].
\end{array}
$$
Rewriting:
$$\begin{array}{ll}
II_{\nu}(\textbf{u},\textbf{v}) & =\bar{\alpha}[a(\alpha l_{\nu_{1}}+\beta m_{\nu_{1}})+b(\alpha l_{\nu_{2}}+\beta m_{\nu_{2}})] \\
              & +\bar{\beta}[a(\alpha m_{\nu_{1}}+\beta n_{\nu_{1}})+b(\alpha m_{\nu_{2}}+\beta n_{\nu_{2}})].
\end{array}
$$
In order for $\textbf{u}\in T_{q}\tilde{M}$ to be an asymptotic direction, we must show that $II_{\nu}(\textbf{u},\textbf{v})=0$.
The last equality above must be satisfied for all $\textbf{v}=(\bar{\alpha},\bar{\beta})\in T_{q}\tilde{M}$, so
$$\left\{\begin{array}{l}
(\alpha l_{\nu_{1}}+\beta m_{\nu_{1}})a+(\alpha l_{\nu_{2}}+\beta m_{\nu_{2}})b=0 \\
(\alpha m_{\nu_{1}}+\beta n_{\nu_{1}})a+(\alpha m_{\nu_{2}}+\beta n_{\nu_{2}})b=0.
\end{array}\right.
$$
Once $a=b=0$ is not a solution, we will have more solutions if and only if
$$
\det\left(
      \begin{array}{cc}
        \alpha l_{\nu_{1}}+\beta m_{\nu_{1}} & \alpha l_{\nu_{2}}+\beta m_{\nu_{2}} \\
        \alpha m_{\nu_{1}}+\beta n_{\nu_{1}} & \alpha m_{\nu_{2}}+\beta n_{\nu_{2}} \\
      \end{array}
    \right)=0,
$$
that is,
$$\alpha^{2}(l_{\nu_{1}}m_{\nu_{2}}-l_{\nu_{2}}m_{\nu_{1}})+\alpha\beta(l_{\nu_{1}}n_{\nu_{2}}-l_{\nu_{2}}n_{\nu_{1}})+\beta^{2}(m_{\nu_{1}}n_{\nu_{2}}-m_{\nu_{2}}n_{\nu_{1}})=0,$$
which is equivalent to what we wanted.
\end{proof}



We know that the curvature parabola $\Delta_{p}$ can be parametrised as in (\ref{parabola}). Each parameter $y\in\mathbb{R}$ corresponds to a unit tangent direction $\textbf{u}=\partial_{x}+y\partial_{y}=(1,y)\in C_{q}$. We shall denote by $y_{\infty}$ the parameter corresponding to the null tangent direction given by $\textbf{u}=\partial_{y}=(0,1)$. For each possibility of $\Delta_{p}$ we define $\eta(y_{\infty})$: when $\Delta_{p}$ is a line or a half-line $\eta(y_{\infty})=\eta'(y_{\infty})=\eta'(y)/|\eta'(y)|$ where $y>0$ is any value such that $\eta'(y)\neq0$. When $\Delta_{p}$ degenerates into a point $\nu$, $\eta(y_{\infty})=\nu$ and $\eta'(y_{\infty})=0$. Finally, in the case where $\Delta_{p}$ is a non-degenerate parabola, $\eta(y_{\infty})$ and $\eta'(y_{\infty})$ are not defined. The previous construction was first considered for corank $1$ surfaces in $\mathbb{R}^{3}$ (see \cite{MartinsBallesteros}).

\begin{lem}\label{lemma-collinear}
A tangent direction in $T_{q}\tilde{M}$ given by a parameter $y\in\mathbb{R}\cup[y_{\infty}]$ is asymptotic if and only if,
$\pi(\eta(y))$ and $\pi(\eta'(y))$ are collinear (if they are defined), where $\pi:N_{p}M\rightarrow E_{p}$ is the canonical projection.
\end{lem}
\begin{proof}
The proof will be split in two cases.
\begin{itemize}
\item[(i)] Consider $y\in\mathbb{R}$. By (\ref{parabola}) and Remark \ref{obs-quebra-do-normal} we have
$$\eta(y)=\sum_{i=1}^{3}(l_{\nu_{i}}+2m_{\nu_{i}}y+n_{\nu_{i}}y^{2})\nu_{i}\ \mbox{and}\ \eta'(y)=\sum_{i=1}^{3}(2m_{\nu_{i}}+2n_{\nu_{i}}y)\nu_{i},$$
in the same conditions as before. Thus,
$$\pi(\eta(y))=\sum_{i=1}^{2}(l_{\nu_{i}}+2m_{\nu_{i}}y+n_{\nu_{i}}y^{2})\nu_{i}\ \mbox{and}\ \pi(\eta'(y))=\sum_{i=1}^{2}(2m_{\nu_{i}}+2n_{\nu_{i}}y)\nu_{i}$$
and these vectors are collinear if and only if the following determinant vanishes: $\det(\textbf{e},\pi(\eta(y)),\pi(\eta'(y)),\nu_{3})=$
$$\begin{array}{cc}
\hspace{-2cm} & =\det\left(                               \begin{array}{cccc}
1 & 0 & 0 & 0 \\
 0 & l_{\nu_{1}}+2m_{\nu_{1}}y+n_{\nu_{1}}y^{2} & l_{\nu_{2}}+2m_{\nu_{2}}y+n_{\nu_{2}}y^{2} & 0 \\
 0 & 2m_{\nu_{1}}+2n_{\nu_{1}}y & 2m_{\nu_{2}}+2n_{\nu_{2}}y & 0 \\
 0 & 0 & 0 & 1 \\
 \end{array}
 \right) \\
 & =2[(l_{\nu_{1}}m_{\nu_{2}}-l_{\nu_{2}}m_{\nu_{1}})+(l_{\nu_{1}}n_{\nu_{2}}-l_{\nu_{2}}n_{\nu_{1}})y+(m_{\nu_{1}}n_{\nu_{2}}-m_{\nu_{2}}n_{\nu_{1}})y^{2}], \end{array}
$$
where $\textbf{e}=(1,0,0,0)\in T_{p}M$. Since the equality above is the same as the one of Lemma \ref{lema-l.a.} for $\textbf{u}=(1,y)\in C_{q}$, the first case is done.
\item[(ii)] On the other hand, when $y=y_{\infty}$, $\eta(y)$ and $\eta'(y)$ are already collinear, if they are defined. Hence, they are collinear if and only if $\Delta_{p}$ is a degenerate parabola, which is equivalent to
$m_{\nu_{1}}n_{\nu_{2}}-m_{\nu_{2}}n_{\nu_{1}}=0$ (See Theorem \ref{teo.2jato/parabola}). Lastly, notice that the equation above is obtained by Lemma \ref{lema-l.a.} with $\textbf{u}=(0,1)$. There is no need to project the vector in $E_{p}$, since $\mathcal{A}ff_{p}\subset E_{p}$.
\end{itemize}
Therefore we conclude the proof.
\end{proof}

Each parameter $y\in\mathbb{R}\cup[y_{\infty}]$ corresponding to an asymptotic direction $\textbf{u}\in T_{q}\tilde{M}$ will also be called an asymptotic direction. Using this nomenclature we can count the number of asymptotic directions for each type of curvature parabola $\Delta_{p}$ (see Figure \ref{asymptotic}).

\begin{itemize}
\item[(i)] When $\Delta_{p}$ is a non degenerate parabola, we have $0$, $1$ or $2$ asymptotic directions according to the position of $p$: ``inside", on or ``outside" of $\pi(\Delta_{p})$, respectively, where  $\pi:N_{p}M\rightarrow E_{p}$ is as before;
\item[(ii)] When $\Delta_{p}$ is a half-line, there are $2$ asymptotic directions $\{y_{\nu},y_{\infty}\}$, with $\eta(y_{\nu})$ being its vertex if $\Delta_{p}$ is not radial and all directions are asymptotic otherwise;
\item[(iii)] When $\Delta_{p}$ is a line, $y_{\infty}$ is the only asymptotic direction if the line is not radial and if it is, all directions are asymptotic;
\item[(iv)] When $\Delta_{p}$ is a point, every $y\in\mathbb{R}\cup[y_{\infty}]$ is an asymptotic direction.
\end{itemize}

\begin{figure}[h!]
\begin{center}
\includegraphics[scale=0.5]{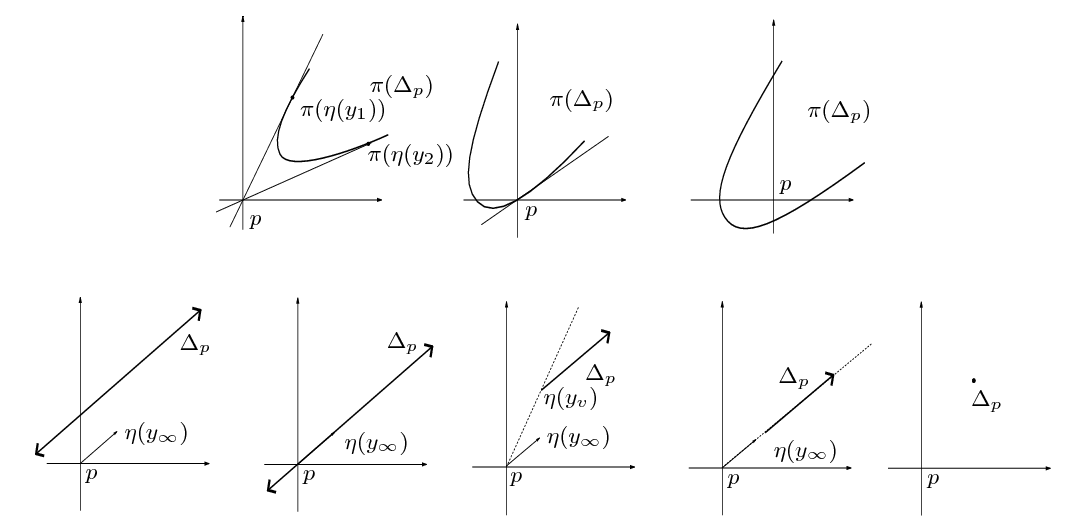}
\caption{Possibilities for $\Delta_{p}$ and the number of asymptotic directions}
\label{asymptotic}
\end{center}
\end{figure}


The following result is a consequence of the previous ones.

\begin{coro}
A normal direction $\nu\in E_{p}$ is a binormal direction if and only if there is an asymptotic direction $y\in\mathbb{R}\cup[y_{\infty}]$ such that
$\nu$ is orthogonal to the subspace spanned by $\pi(\eta(y))$ and $\pi(\eta'(y))$, where $\pi:N_{p}M\rightarrow E_{p}$ is the orthogonal projection.
\end{coro}

\begin{definition}
Given a binormal direction $\nu\in E_{p}$, the hyperplane through $p$ and orthogonal to $\nu$ is called an \emph{osculating hyperplane} to $M$ at $p$.
\end{definition}

The number of binormal directions and of osculating hyperplanes
depends on the type of the curvature parabola. Using the Corollary
above, we can calculate this number. If $\Delta_{p}$ is a non
degenerate parabola, we have one $0$, $1$ or $2$ binormal
directions, one for each asymptotic direction. When $\Delta_{p}$ is
a half-line, there are three possibilities: if the line that
contains $\Delta_{p}$ is not radial, there are $2$ binormal
directions. Otherwise, we have one if the vertex of the parabola is
not $p$ and if it is, all directions are binormal. When $\Delta_{p}$
is a line or a point different from $p$, we have one binormal
direction. Finally, if $\Delta_{p}=\{p\}$ all directions are
binormal.

\begin{definition}\label{pointstypes}
Given a surface $M\subset\mathbb{R}^{4}$ with corank $1$ singularity at $p\in M$. The point $p$ is called:
\begin{itemize}
    \item[(i)] \emph{elliptic} if there are no asymptotic directions at $p$;
    \item[(ii)] \emph{hyperbolic} if there are two asymptotic directions at $p$;
    \item[(iii)] \emph{parabolic} if there is one asymptotic direction at $p$;
    \item[(iv)] \emph{inflection} if there are an infinite number of asymptotic directions at $p$.
\end{itemize}
\end{definition}

From Definition \ref{pointstypes} and the previous counting of the number of asymptotic directions, $p$ is parabolic, hyperbolic or elliptic if and only if $p\in M_{2}\cup M_{3}$. On the other hand, $p$ is an inflection point if and only if $p\in M_{1}\cup M_{0}$.

The next result proves that the geometry of a corank $1$ surface in $\mathbb{R}^{4}$ is strongly related to the geometry of the associated regular surface $S\subset\mathbb{R}^{4}$ as defined previously.

\begin{teo}\label{teorelation}
Let $M\subset\mathbb{R}^{4}$ be a surface with corank $1$
singularity at $p\in M$ and $S\subset\mathbb{R}^{4}$ the regular
surface associated to $M$.
\begin{itemize}
    \item[(i)] A direction $\textbf{u}\in T_{q}\tilde{M}$ is an asymptotic direction of $M$ if and only if it is also an asymptotic direction of the associated regular surface $S\subset\mathbb{R}^{4}$;
   \item[(ii)] A direction $\nu\in N_{p}M$ is a binormal direction of $M$ if and only if $\texttt{p}\circ
   \xi^{-1}(\nu)\in N_{\texttt{p}\circ\xi^{-1}(p)}S$ is a binormal direction of $S$.
    \item[(iii)] The point $p$ is an elliptic/hyperbolic/parabolic/inflection point if and only if $\texttt{p}\circ\xi^{-1}(p)\in S$ is an elliptic/hyperbolic/parabolic/inflection point, respectively.
\end{itemize}
\end{teo}
\begin{proof}
The first two items follow directly from the definitions and from Lemma \ref{lema-l.a.}.
The third part follows immediately from Definition \ref{pointstypes} and the first item.
\end{proof}


In \cite{Rieger}, the authors provide a classification of $\mathcal{A}$-simple map germs $f:(\mathbb{R}^{2},0)\rightarrow(\mathbb{R}^{4},0)$.
The singularity $I_{k}$, $k\geqslant1$, given by the $\mathcal{A}$-normal form $(x,y)\mapsto(x,xy,y^{2},y^{2k+1})$ has the following property: every germ $\mathcal{A}$-equivalent to it prarametrises a corank $1$ surface in $\mathbb{R}^{4}$ such that the corresponding curvature parabola is a non degenerate parabola. Moreover, $I_{k}$ are the only singularities having this property. Hence, every map germ $\mathcal{A}$-equivalent to $I_{k}$ is $\mathcal{R}^{2}\times\mathcal{O}(4)$-equivalent to the first normal form in Lemma \ref{lemma2jato2}.

\begin{prop}
Consider the $\mathcal{R}^{2}\times\mathcal{O}(4)$ normal form of the singularity $I_{k}$ given by $f:(\mathbb{R}^{2},0)\rightarrow(\mathbb{R}^{4},0)$ where $f(x,y)=(x,xy+p(x,y),b_{20}x^{2}+b_{11}xy+b_{02}y^{2}+q(x,y),c_{20}x^{2}+r(x,y))$ with $b_{02}>0$ and $p,q,r\in\mathcal{M}_{2}^{3}$ (see Theorem \ref{lemma2jato2}). Then, the singularity $I_{k}$ is hyperbolic, parabolic or elliptic if and only if $b_{20}$ is positive, zero or negative, respectively.
\end{prop}
\begin{proof}
The curvature parabola $\Delta_{p}$ is a non degenerate parabola that can be parametrised by $\eta(y)=(0,2y,2b_{20}+2b_{11}y+2b_{02}y^{2},2c_{20})$ and $E_{p}$ is the $YZ$-plane. From Lemma \ref{lemma-collinear},  $y\in\mathbb{R}$ is an asymptotic direction iff $\pi(\eta(y))$ and $\pi(\eta'(y))$ are collinear. This condition is equivalent to
$$\det(\textbf{e},\pi(\eta(y)),\pi(\eta'(y)),\nu_{3})=4(b_{02}y^{2}-b_{20})=0,$$
where $T_{p}M=[\textbf{e}]$ and $E_{p}^{\perp}=[\nu_{3}]$
Hence, $y\in\mathbb{R}$ is asymptotic iff $y=\pm\sqrt{b_{20}b_{02}}/b_{02}$. Since $b_{02}>0$, we shall have $0$, $1$ or $2$ asymptotic directions according to whether $b_{20}$ is negative, zero or positive, respectively.
\end{proof}


\subsection{Umbilic curvature}
In the literature, one can find invariants called umbilic curvature. It can be found in articles about regular surfaces in $\mathbb{R}^{5}$ (\cite{Costa/Fuster/Moraes}, for example) and also about singular surfaces in $\mathbb{R}^{3}$ (\cite{MartinsBallesteros,MartinsSaji}). Geometrically speaking, the definition of these invariants are very similar in both cases: they measure, in a sense, the distance between the respective locus of curvature of the surface at the point and the point itself. In our case, the umbilic curvature will perform a resemblant role. However, it will be a second order invariant, meaning that it will depend on the second order information of the surface (in our case, this is provided by the map $g:\tilde{M}\rightarrow\mathbb{R}^{4}$ considered before). In other words, the umbilic curvature is a $\mathcal{R}^{2}\times\mathcal{O}(4)$-invariant.

\begin{definition}\label{def.curvatura}
The non-negative number
$$\kappa_{u}(p)=d(p,\mathcal{A}ff_{p})$$
is called the \emph{umbilic curvature} of $M$ at $p$.
\end{definition}

Let $\{\nu_{1},\nu_{2},\nu_{3}\}$ be an orthonormal frame of $N_{p}M$ such as in Remark \ref{obs-quebra-do-normal}. We can write
 $$II(\textbf{u},\textbf{u})=II_{\nu_{1}}(\textbf{u},\textbf{u})\nu_{1}+II_{\nu_{2}}(\textbf{u},\textbf{u})\nu_{2}+II_{\nu_{3}}(\textbf{u},\textbf{u})\nu_{3},$$
for $\textbf{u}\in C_{q}\subset T_{q}\tilde{M}$. When $\Delta_{p}$ is a non degenerate parabola, $II(\textbf{u},\textbf{u})\in\Delta_{p}\subset\mathcal{A}ff_{p}$. Then $II_{\nu_{3}}(\textbf{u},\textbf{u})$ does not depend on $\textbf{u}$ up to sign, in the sense that $II_{\nu_{3}}(\textbf{u},\textbf{u})$ only tells us about the position of $\Delta_{p}$ in $N_{p}M$ (see Remark \ref{obs-quebra-do-normal}). When $\Delta_{p}$ degenerates, the expression above does not depend on $II_{\nu_{3}}(\textbf{u},\textbf{u})\nu_{3}$, since $\Delta_{p}\subset E_{p}$. First, suppose that $\Delta_{p}$ is not a point. The asymptotic direction $y_{\infty}$ is well defined and the binormal direction $\nu_{\infty}$, called \emph{infinite binormal direction}, is such that $\{\eta(y_{\infty}),\nu_{\infty}\}$ is an orthonormal positively oriented frame of $E_{p}$. If $\Delta_{p}$ is a point which is not $p$, $\eta(y)$ is constant we take the orthonormal positively oriented frame $\{\nu,\eta(y)/|\eta(y)|\}$ of $E_{p}$ with $\nu$ a binormal direction. The frames above are called \emph{adapted frames} of $E_{p}$ (see Figure \ref{frame}). If $\Delta_{p}=\{p\}$ or a non degenerate parabola, then any frame can be considered an adapted frame.

\begin{prop}
Let $\{\nu_{1},\nu_{2},\nu_{3}\}$ be an othonormal frame of $N_{p}M$ satisfying that $\{\nu_{1},\nu_{2}\}$ is an adapted frame of $E_{p}$. Then for $\textbf{u}\in C_{q}$,
 $$\kappa_{u}(p)=\left\{\begin{array}{cl}
    | II_{\nu_{3}}(\textbf{u},\textbf{u})|, & \mbox{if}\ \Delta_{p}\ \mbox{does not degenerate} \\
    |II_{\nu_{2}}(\textbf{u},\textbf{u})|, & \mbox{if}\ \Delta_{p}\ \mbox{ degenerates}
  \end{array}\right.
$$
\end{prop}
\begin{proof}
When $\Delta_{p}$ is a non degenerate parabola, $d(p,\mathcal{A}ff_{p})=d(p,II_{\nu_{3}}(\textbf{u},\textbf{u}))= |II_{\nu_{3}}(\textbf{u},\textbf{u})|$ for $\textbf{u}\in C_{q}$. On the other hand, when $\Delta_{p}$ degenerates, $II_{\nu_{3}}(\textbf{u},\textbf{u})=0$ since $\Delta_{p}\subset E_{p}$. Also, $II(\textbf{u},\textbf{u})$ does not depend on $II_{\nu_{2}}(\textbf{u},\textbf{u})\nu_{2}$, for $\textbf{u}\in C_{q}$, up to sign. Therefore, $d(p,\mathcal{A}ff_{p})=d(p,II_{\nu_{2}}(\textbf{u},\textbf{u}))= |II_{\nu_{2}}(\textbf{u},\textbf{u})|$.
\end{proof}


\begin{figure}[h!]
\begin{center}
\includegraphics[scale=0.45]{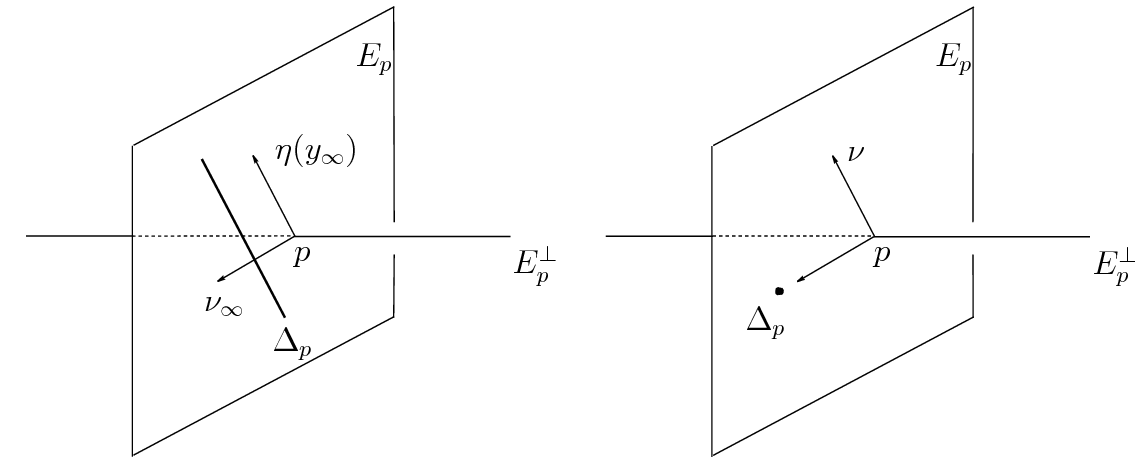}
\caption{Adapted frame of $E_{p}$}
\label{frame}
\end{center}
\end{figure}

\begin{rem}
When $p\in M_{1}$ for every choice of $E_{p}$ we have a different adapted frame. However, neither the choice of $E_{p}$ nor of the adapted frame changes the umbilic curvature. Indeed, suppose that $\Delta_{p}$ is a radial half-line or a radial line. Let $\Gamma_{1},\Gamma_{2}\subset N_{p}M$ be planes through $p$ such that $\Delta_{p}\subset\Gamma_{1},\Gamma_{2}$. Consider the adapted frames $\{\eta(y_{\infty}),\nu_{\infty}^{1}\}$ and $\{\eta(y_{\infty}),\nu_{\infty}^{2}\}$ of $\Gamma_{1}$ and $\Gamma_{2}$, respectively. Since $\Delta_{p}$ is contained in the line given by the direction of $\eta(y_{\infty})$,
$$II_{\nu_{\infty}^{1}}(\textbf{u},\textbf{u})=II_{\nu_{\infty}^{2}}(\textbf{u},\textbf{u})=0,$$
for all $\textbf{u}\in C_{q}$. Besides, from Definition
\ref{def.curvatura} $\kappa_{u}(p)=0$, once $p\in\mathcal{A}ff_{p}$.
In the case when $\Delta_{p}$ is a point other then $p$, the proof
is analogous.
\end{rem}

\begin{coro}
Let $M\subset\mathbb{R}^{4}$ be a corank $1$ surface at $p$ such that $\Delta_{p}$ is a non degenerate parabola and $\{\nu_{1},\nu_{2},\nu_{3}\}$ an orthonormal frame of $N_{p}M$ as in Remark \ref{obs-quebra-do-normal}. Then, $$\kappa_{u}(p)=\frac{|II_{\nu_{3}}(\textbf{u},\textbf{u})|}{I(\textbf{u},\textbf{u})}=|\mbox{proj}_{\nu_{3}}\eta(y)|=|\langle \eta(y),\nu_{3}\rangle|,$$
for any $\textbf{u}\in T_{q}\tilde{M}$.
\end{coro}
\begin{proof}
Let $\textbf{u}\in T_{q}\tilde{M}$. The first equality follows from
$$|II_{\nu_{3}}(\textbf{u},\textbf{u})|=|\textbf{u}|^{2}\left|II_{\nu_{3}}\left(\frac{\textbf{u}}{|\textbf{u}|},\frac{\textbf{u}}{|\textbf{u}|}\right)\right|=I(\textbf{u},\textbf{u})\kappa_{u}(p).$$
Since $\Delta_{p}\subset\mathcal{A}ff_{p}$, the last two equations follow from the definition of the umbilic curvature.
\end{proof}

\begin{coro}
When $\Delta_{p}$ is a half-line or a line, $\kappa_{u}(p)$ is the length of the projection of $\Delta_{p}$ on the direction given by an infinite binormal direction, then we have the following formulas:
$$\kappa_{u}(p)=\left|\left\langle\eta(y),\frac{\textbf{e}\times\eta'(y)\times\nu_{3}}{|\eta'(y)|}\right\rangle\right|=\frac{1}{|\eta'(y)|}|\det(\textbf{e},\eta(y),\eta'(y),\nu_{3})|$$
and
$$\kappa_{u}(p)=\frac{|\textbf{e}\times\eta(y)\times\eta'(u)|}{|\eta'(y)|},$$
where $T_{p}M=[\textbf{e}]$, $E_{p}^{\perp}=[\nu_{3}]$ and $\eta'(y)\neq0$ (see Figure \ref{frame2}).
Also, for any $\textbf{u}\in T_{q}\tilde{M}$ and $\{\nu_{1},\nu_{2},\nu_{3}\}$ orthonormal frame of $N_{p}M$ satisfying that $\{\nu_{1},\nu_{2}\}$ is an adapted frame of $E_{p}$, we have:
$$\kappa_{u}(p)=\frac{|II_{\nu_{2}}(\textbf{u},\textbf{u})|}{I(\textbf{u},\textbf{u})}.$$
\begin{figure}[h!]
\begin{center}
\includegraphics[scale=0.35]{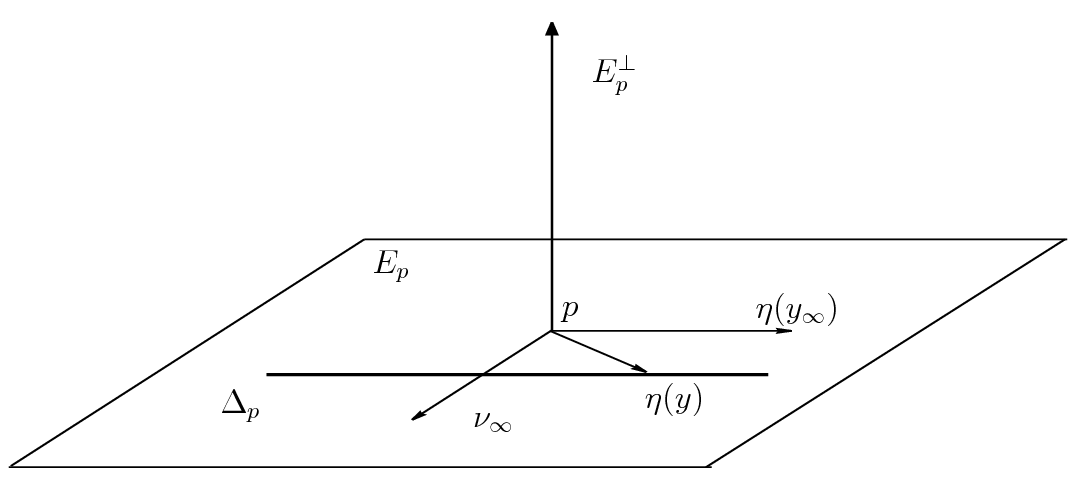}
\caption{}
\label{frame2}
\end{center}
\end{figure}
\end{coro}


The umbilic curvature does not depend on the choice of the adapted frame of $E_{p}$, nor the frame of $N_{p}M$, nor the parametrisation of $\eta$, nor the choice of local coordinates of $\tilde{M}$. However, it depends on the map $g:\tilde{M}\rightarrow\mathbb{R}^{4}$ that parametrises $M$. Indeed, the surface given by $f(x,y)=(x,y^{2},y^{3},x^{2}y)$ has $\kappa_{u}(0)=0$. On the other hand, the same surface given by $\bar{f}(x,y)=(x,(y^{3}+x)^{2},(y^{3}+x)^{3},(y^{3}+x)^{2}y)$ has $\bar{\kappa}_{u}(0)=2$.

\begin{lem}
Let $M\subset\mathbb{R}^{4}$ be a corank $1$ surface at $p\in M$. The following holds:
\begin{enumerate}
    \item[(i)] When $\Delta_{p}$ is a non degenerate parabola, $p\in M_{3}$ iff $\kappa_{u}(p)\neq0$;
    \item[(ii)] When $\Delta_{p}$ is a half-line or a line, $p\in M_{2}$ iff $\kappa_{u}(p)\neq0$;
    \item[(iii)] When $\Delta_{p}$ is a point, $p\in M_{1}$ iff $\kappa_{u}(p)\neq0$.
\end{enumerate}
\end{lem}
\begin{proof}
The proof is a straightforward calculation using Lemma \ref{lemma/tipo.de.pontos}. For completeness, we shall prove the first case using the normal form from Lemma \ref{lemma2jato2}. Taking the local parametrisation $f(x,y)=(x,a_{11}xy+p(x,y),b_{20}x^{2}+b_{11}xy+b_{02}y^{2}+q(x,y),c_{20}x^{2}+r(x,y))$, with $p,q,r\in\mathcal{M}_{2}^{3}$ and $b_{02}\neq0$, the matrix of the second fundamental form is given by
$$
\left(
\begin{array}{ccc}
    0 & a_{11} & 0  \\
    2b_{20} & b_{11} & 2b_{02} \\
    2c_{20} & 0 & 0
\end{array}
\right).
$$
Then, $p\in M_{3}$ iff $c_{20}\neq0$, otherwise $p\in M_{2}$. The plane $\mathcal{A}ff_{p}$ is the one parallel to the plane $YZ$ through $(0,0,0,2c_{20})$. Thus, $\kappa_{u}(p)=2|c_{20}|$ and $p\in M_{3}$ if and only if $\kappa_{u}(p)\neq0$.
\end{proof}


\subsection{Flat geometry} In this section we shall deal with the contact between a corank $1$ surface in $\mathbb{R}^{4}$ and hyperplanes. Such contact is measured by the singularities of the height function defined in the following way: given a local parametrisation of a corank $1$ surface $M$ at $p$ (we shall assume $p$ as the origin), $f:(\mathbb{R}^{2},0)\rightarrow(\mathbb{R}^{4},0)$ the \emph{height function} $h_{v}:(\mathbb{R}^{2},0)\rightarrow(\mathbb{R},0)$ is given by $h_{v}(x,y)=\langle f(x,y),v\rangle$, $v\in\mathbb{S}^{3}$.

\begin{lem}\label{lemma-equiv}
Let $M\subset\mathbb{R}^{4}$ be a corank $1$ surface at $p\in M$ and a nontrivial vector $\nu\in N_{p}M$. The quadratic forms $II_{\nu}(p)$ and $\mathcal{H}(h_{\nu})(p)$ are equivalent (in the sense that we may find coordinate systems in which they coincide), where $\mathcal{H}(h_{\nu})(p)$ denotes the Hessian matrix of $h_{\nu}$ at $p$.
\end{lem}
\begin{proof}
Let $f:(\mathbb{R}^{2},0)\rightarrow(\mathbb{R}^{4},0)$ be a local parametrisation of $M$ at $p$, given by
$f(x,y)=(x,f_{2}(x,y),f_{3}(x,y),f_{4}(x,y))$ with $f_{i}\in \mathcal{M}_{2}^{2}$, $i=2,3,4$. We may write
$$f_{2}(x,y)=\sum_{i+j=2}a_{ij}x^{i}y^{j},\ f_{3}(x,y)=\sum_{i+j=2}b_{ij}x^{i}y^{j}\ \mbox{and}\ f_{4}(x,y)=\sum_{i+j=2}c_{ij}x^{i}y^{j}.
$$
Also, let $\nu=(0,v_{2},v_{3},v_{4})\in N_{p}M$. Thus, $h_{\nu}(x,y)=xv_{1}+f_{2}(x,y)v_{2}+f_{3}(x,y)v_{3}+f_{4}(x,y)v_{4}$ and
$$\mathcal{H}(h_{\nu})(p)=\left(\begin{array}{cc}
2(a_{20}v_{2}+b_{20}v_{3}+c_{20}v_{4}) & a_{11}v_{2}+b_{11}v_{3}+c_{11}v_{4}\\
a_{11}v_{2}+b_{11}v_{3}+c_{11}v_{4} & 2(a_{02}v_{2}+b_{02}v_{3}+c_{02}v_{4}).
\end{array}\right)
$$
On the other hand, $II_{\nu}(\textbf{u},\textbf{u})=x^{2}l_{\nu}+2xym_{\nu}+y^{2}n_{\nu}$, for $\textbf{u}=(x,y)\in T_{q}\tilde{M}$, where
$$ \begin{array}{c}
l_{\nu}=\langle f_{xx},\nu\rangle=2(a_{20}v_{2}+b_{20}v_{3}+c_{20}v_{4}),\  m_{\nu}=\langle f_{xy},\nu\rangle=a_{11}v_{2}+b_{11}v_{3}+c_{11}v_{4},\\
n_{\nu}=\langle f_{yy},\nu\rangle=2(a_{02}v_{2}+b_{02}v_{3}+c_{02}v_{4}).
\end{array}$$
Hence, the matrix of $II_{\nu}(p)$ is
$$
II_{\nu}(p)=\left(\begin{array}{cc}
2(a_{20}v_{2}+b_{20}v_{3}+c_{20}v_{4}) & a_{11}v_{2}+b_{11}v_{3}+c_{11}v_{4}\\
a_{11}v_{2}+b_{11}v_{3}+c_{11}v_{4} & 2(a_{02}v_{2}+b_{02}v_{3}+c_{02}v_{4}).
\end{array}\right)=\mathcal{H}(h_{\nu})(p).
$$
\end{proof}


\begin{teo}\label{teo-degenerada}
Let $M\subset\mathbb{R}^{4}$ be a corank $1$ surface at $p\in M$. The height function $h_{\nu}$ is singular at $p$ if and only if $\nu\in N_{p}M$. Furthermore,
\begin{itemize}
\item[(i)] Assume $\Delta_{p}$ is not a point. Then, $h_{\nu}$ has a degenerate singularity at $p$ iff $\nu$ is a degenerate direction. When $\Delta_{p}$ is a non-degenerate parabola, $h_{\nu}$ has a corank $2$ singularity at $p$ iff $\nu\in E_{p}^{\perp}$ and $\kappa_{u}(p)=0$. When $\Delta_{p}$ is a line or a half-line, there are two possibilities. If $\kappa_{u}(p)\neq0$, the only direction such that $h_{\nu}$ has a corank $2$ singularity at $p$ is given by $E_{p}^{\perp}$. However, if $\kappa_{u}(p)=0$, $h_{\nu}$ has a corank $2$ singularity at $p$ for all directions on $\mathcal{A}ff_{p}^{\perp}$;
\item[(ii)] Assume $\Delta_{p}$ is a point. Then $h_{\nu}$ has a degenerate singularity at $p$ for all $\nu\in N_{p}M$. Also, if $\kappa_{u}(p)\neq0$, $h_{\nu}$ has a corank $2$ singularity at $p$ for all directions on the plane orthogonal to he line through $p$ containing $\Delta_{p}$. If $\kappa_{u}(p)=0$ the singularity is of corank $2$ for all $\nu\in N_{p}M$.
\end{itemize}
\end{teo}
\begin{proof}
It is obvious that $h_{\nu}$ is singular iff $\nu\in N_{p}M$. When $\Delta_{p}$ is a non degenerate parabola, taking the parametrisation given in Lemma \ref{lemma2jato2},
the second fundamental form in the direction $\nu=(0,v_{2},v_{3},v_{4})\in N_{p}M$ is given by
$$
II_{\nu}(\textbf{u},\textbf{v})=\langle (0,\alpha_{1}\beta_{2}+\alpha_{2}\beta_{1},2b_{20}\alpha_{1}\alpha_{2}+b_{11}(\alpha_{1}\beta_{2}+\alpha_{2}\beta_{1})+2b_{02}\beta_{1}\beta_{2},2c_{20}\alpha_{1}\alpha_{2}),\nu\rangle,
$$
where $\textbf{u}=(\alpha_{1},\beta_{1}),\textbf{v}=(\alpha_{2},\beta_{2})\in T_{q}\tilde{M}$. Hence, $\nu$ is a degenerate direction if and only if $II_{\nu}(\textbf{u},\textbf{v})=0$ for all $\textbf{v}\in T_{q}\tilde{M}$, which is equivalent to $v_{2}=-b_{11}v_{3}\pm2\sqrt{b_{20}b_{02}v_{3}^{2}+b_{02}c_{20}v_{3}v_{4}}$.
On the other hand, the height function is given by
$$
h_{\nu}(x,y)=(xy+p(x,y))v_{2}+(b_{20}x^{2}+b_{11}xy+b_{02}y^{2}+q(x,y))v_{3}+(c_{20}x^{2}+q(x,y))v_{4},
$$
with $b_{02}>0$ and $p,q,r\in\mathcal{M}_{2}^{3}$ and
the Hessian matrix of $h_{\nu}$ at $p$ is
$$\mathcal{H}(h_{\nu})(p)=\left(\begin{array}{cc}
2(b_{20}v_{3}+c_{20}v_{4}) & v_{2}+b_{11}v_{3}\\
v_{2}+b_{11}v_{3} & 2b_{02}v_{3}.
\end{array}\right).
$$
Therefore, $h_{\nu}$ has a degenerate singularity (non Morse singularity) at $p$ if and only if $\det(\mathcal{H}(h_{\nu})(p))=0$, and it is equivalent to $v_{2}=-b_{11}v_{3}\pm2\sqrt{b_{20}b_{02}v_{3}^{2}+b_{02}c_{20}v_{3}v_{4}}$.
The Hessian matrix vanishes iff $\nu\in E_{p}^{\perp}$ and $\kappa_{u}(p)=0$, since $E_{p}^{\perp}=[v_{4}]$ and $\kappa_{u}(p)=2|c_{20}|$.
When $\Delta_{p}$ is a half-line or a line, the proof is analogous.
Finally, when $\Delta_{p}$ is a point, $\Delta_{p}$ can be parametrised by $\eta(y)=(0,0,2b_{20},0)$, $\kappa_{u}(p)=2|b_{20}|$ and the height function is given by $h_{\nu}(x,y)=p(x,y)v_{2}+(b_{20}x^{2}+q(x,y))v_{3}+r(x,y)v_{4}.$ It degenerates for all $\nu\in N_{p}M$ since
$$\mathcal{H}(h_{\nu})(p)=\left(
                                             \begin{array}{cc}
                                               2b_{20}v_{3} & 0 \\
                                               0 & 0 \\
                                             \end{array}
                                           \right).
$$
Furthermore, $h_{\nu}$ has a corank $2$ singularity iff $\kappa_{u}(p)v_{3}=0$. Therefore, if $\kappa_{u}(p)\neq0$, we have $v_{3}=0$ and for all $\nu$ orthogonal to $(0,0,2b_{20},0)$, that is, $\nu$ orthogonal to the plane through $p$ containing $\Delta_{p}$, $h_{\nu}$ has a corank $2$ singularity. If $\kappa_{u}(p)=0$, $h_{\nu}$ has a corank $2$ singularity at $p$ for all $\nu\in N_{p}M$.
\end{proof}

It follows from Theorem \ref{teo-degenerada} that the subset of degenerate directions $\nu\in N_{p}M$ is exactly the subset of degenerate directions of the height function (a cone, denoted by $\mathcal{C}_{p}\subset N_{p}M$ that may degenerate). Besides, the binormal directions are obtained by the intersection of the cone of degenerate directions, $\mathcal{C}_{p}$, with $E_{p}$. Here lies the difference between the notion of a binormal direction for regular surfaces in $\mathbb{R}^{5}$ and our case. In the first one, a vector $\nu$ is a binormal direction iff the height function $h_{\nu}$ has a singularity more degenerate than $A_{3}$. For our case, however, in order to be a binormal direction, the vector $\nu$ must lie in the intersection of $E_{p}$ and the cone of degenerate directions of the height function $h_{\nu}$. Therefore, our definition is related to the second order information of the surface.

In \cite{Benedini/Sinha} the first two authors showed that the number of asymptotic and binormal directions of a regular surface in $\mathbb{R}^{4}$ and a corank $1$ surface in $\mathbb{R}^{3}$ given by the projection of the first one in a tangent direction coincide. Nevertheless, this coincidence does not happen here. 

\begin{coro}
Let $\textbf{u}\in T_{q}\tilde{M}$ be an asymptotic direction given by a parameter $y\in\mathbb{R}\cup[y_{\infty}]$. Then, the cone whose basis is the curvature parabola is perpendicular to the cone $\mathcal{C}_{p}$ of degenerate directions at $p$.
\end{coro}
\begin{proof}
The tangent direction $\textbf{u}\in T_{q}\tilde{M}$ is an asymptotic direction given by the parameter $y\in\mathbb{R}\cup[y_{\infty}]$ if and only if, there is a binormal direction $\nu\in\mathcal{C}_{p}\cap E_{p}$ such that for all tangent vectors $\textbf{v}\in T_{q}\tilde{M}$,
$$II_{\nu}(\textbf{u},\textbf{v})=\langle II(\textbf{u},\textbf{v}),\nu\rangle=0.$$
In particular, for $\textbf{v}=\textbf{u}$, this means that
$\langle\eta(y),\nu\rangle=0$. In other words, the line given by
$\eta(y)$ and $p$ (contained in the cone whose basis is
$\Delta_{p}$) is perpendicular to $\mathcal{C}_{p}$.
\end{proof}

\begin{coro}
Let  $M\subset\mathbb{R}^{4}$ be a surface as in Theorem \ref{teo-degenerada}. A hyperplane has a degenerate contact with $M$ at $p$ if and only if, it is orthogonal to a degenerate direction $\nu\in N_{p}M$. In particular, if $\nu$ is a binormal direction, the hyperplane is an oculating hyperplane.
\end{coro}


\begin{thebibliography}{22}

\bibitem{BruceNogueira} {\sc{J. W. Bruce and A. C. Nogueira}} {\it Surfaces in $\mathbb{R}^{4}$ and duality}. Quart. J.
Math. Oxford Ser. 49 (1998), 433--443.

\bibitem{BruceTari} {\sc{J. W. Bruce and F. Tari}} {\it Families of surfaces in $\mathbb{R}^{4}$}. Proc. Edinb. Math. Soc. (2) 45 (2002), no. 1, 181--203.

\bibitem{BruceWest} {\sc{J. W. Bruce and J. M. West}} {\it Functions on a crosscap}. {\it Math. Proc. Cambridge Philos. Soc.} 123 (1998), 19--39.

\bibitem{Benedini/Sinha} {\sc P. Benedini Riul and R. Oset Sinha} {\it A relation between the curvature ellipse and the curvature parabola}. arXiv:1708.04651, (2017).

\bibitem{Costa/Fuster/Moraes} {\sc{S. I. R. Costa, M. S. Moraes and M. C. Romero Fuster}} {\it Geometric contact of surfaces immersed in $\mathbb{R}^{n}$, $n\geqslant5$}. Differential Geom, Appl. 27 (2009), 442--454.

\bibitem{DiasTari} {\sc{F. S. Dias and F. Tari}} {\it On the geometry of the cross-cap in Minkowski 3-
space and binary differential equations}. Tohoku Math. J. (2) 68, (2016), no.2, 293--328.

\bibitem{FukuiHasegawa} {\sc{T. Fukui and M. Hasegawa}} {\it Fronts of Whitney umbrella – a differential geometric approach via blowing up}. J. Singul. 4 (2012), 35--67.

\bibitem{GarciaMochidaFusterRuas} {\sc{R. Garcia, D. K. H. Mochida, M. C. Romero Fuster and M. A. S. Ruas}} {\it Inflection points and topology of surfaces in 4-space}. Trans. Amer. Math. Soc. 352 (2000), 3029--3043.

\bibitem{HHNUY} {\sc{M. Hasegawa, A. Honda, K. Naokawa, M. Umehara and K. Yamada}} {\it Intrinsic
invariants of cross caps}. Selecta Math. (N.S.) 20 (2014), no. 3, 769--785.

\bibitem{HHNSUY} {\sc{M. Hasegawa, A. Honda, K. Naokawa, K. Saji, M. Humehara and K. Yamada}} {\it Intrinsic properties of surfaces with singularities}. Internat. J. Math. 26 (2015), no. 4, 1540008, 34 pp.

\bibitem{Livro} {\sc{S. Izumiya, M. C. Romero Fuster, M. A. S. Ruas and F. Tari}} {\it Differential Geometry from Singularity Theory Viewpoint}. World Scientific Publishing Co Pte Ltd, Singapore (2015).

\bibitem{Rieger} {\sc{C. Klotz, O. Pop and J. H. Rieger}} {\it Real double-points of deformations of $\mathcal{A}$-simple map-germs from $\mathbb{R}^{n}$ to
$\mathbb{R}^{2n}$}. Math. Proc. Camb. Phil. Soc. (2007), 142-341.

\bibitem{KRSUY} {\sc{M. Kokubu, W. Rossman, K. Saji, M. Umehara and K. Yamada,}} {\it Singularities
of flat fronts in hyperbolic space}. Pacific J. Math. 221 (2005), no. 2, 303--351.

\bibitem{Little} {\sc{J. A. Little}} {\it On singularities of submanifolds of higher dimensional Euclidean
spaces}. Ann. Mat. Pura Appl. 83 (4) (1969), 261--335.

\bibitem {MartinsBallesteros} {\sc{L. F. Martins and J. J. Nu\~{n}o-Ballesteros,}} {\it{Contact properties of surfaces in $\mathbb{R}^{3}$ with corank $1$
    singularities}}. Tohoku Math. J. 67 (2015), 105--124.

\bibitem{MartinsSaji} {\sc{L. F. Martins and K. Saji,}} {\it Geometric invariants of cuspidal edges}. Canadian J. Math 68 (2016), no. 2, 445--462.

\bibitem{MartinsSaji2} {\sc{L. F. Martins and K. Saji,}} {\it Geometry of cuspidal edges with boundary}. Topology and its applications, v. 234 (2018), 209--219.

\bibitem{Mendes} {\sc{R. Mendes and J. J. Nu\~{n}o-Ballesteros,}} {\it Knots and the topology of singular surfaces in $\mathbb{R}^{4}$}. Contemporary Mathematics, vol.675 (2016).

\bibitem{Mochida/Fuster/Ruas} {\sc{D. K. H. Mochida, M. C. Romero Fuster and M. A. S. Ruas,}} {\it Inflection points and nonsigular embeddings of surfaces in $\mathbb{R}^{5}$}. Rocky Mountain J. Math, 33 (2003), 995--1009.

\bibitem{Moraes/Fuster} {\sc{S. M. Moraes and M. C. Romero Fuster,}} {\it Convexity and semiumbilicity for surfaces in $\mathbb{R}^{5}$}. Differential geometry, Valencia 2001, World Sci. Publ., River Edge, NJ (2002), 222--234.


\bibitem{MochidaFusterRuas} {\sc{D. K. H. Mochida, M. C. Romero Fuster and M. A. S. Ruas,}} {\it The geometry of surfaces in $4$-space from a contact viewpoint}.

\bibitem{MochidaFusterRuas2} {\sc{D. K. H. Mochida, M. C. Romero Fuster and M. A. S. Ruas,}} {\it Osculating
hyperplanes and asymptotic directions of codimension two submanifolds of Euclidean spaces}. Geom. Dedicata 77 (1999), 305--315.



\bibitem{NUY} {\sc{K. Naokawa, M. Umehara and K. Yamada,}} {\it Isometric deformations of cuspidal
edges}. Tohoku Math. J. 68, (2016), 73--90.

\bibitem{BallesterosTari} {\sc{J. J. Nu\~{n}o-Ballesteros and F. Tari,}} {\it Surfaces in $\mathbb{R}^{4}$ and their projections to $3$-spaces}. Proc. Roy. Soc. Edinburgh Sect. A, 137 (2007), 1313--1328.

\bibitem{OsetSinhaTari} {\sc{R. Oset Sinha and F. Tari,}} {\it Projections of surfaces in $\mathbb{R}^{4}$ to $\mathbb{R}^{3}$ and the geometry of their singular
   images}. Rev. Mat. Iberoam. 32 (2015), no. 1, 33--50.

\bibitem{OsetSinhaTari2} {\sc{R. Oset Sinha and F. Tari,}} {\it Flat geometry of cuspidal edges}. To appear in Osaka J. Math.

\bibitem{RomeroFuster} {\sc{M. C. Romero Fuster,}} {\it Semiumbilics and geometrical dynamics on surfaces
in $4$-spaces}. Real and complex singularities, Contemp. Math.,
354, Amer. Math. Soc., Providence, RI. (2004) 259--276.

\bibitem{Fuster/Ruas/Tari} {\sc{M. C. Romero Fuster,  M. A. S. Ruas and F. Tari,}} {\it Asymptotic curves on surfaces in $\mathbb{R}^{5}$}. Communications in Contemporary Maths. 10 (2008), 1--27.

\bibitem{SUY} {\sc{K. Saji, M. Umehara, and K. Yamada,}} {\it The geometry of fronts}. Ann. of Math (2) 169 (2009), 491--529.





\end{thebibliography}
\end{document}